\documentclass[12pt,a4paper]{article}

\usepackage{amsmath}
\usepackage{amsfonts}
\usepackage{amssymb}

\usepackage{color}

\usepackage{amsthm}
\usepackage{graphicx}
\usepackage[T1]{fontenc}     
\usepackage{graphicx}
\usepackage{epstopdf}
\usepackage{graphics}
\usepackage{subfigure}
\usepackage{psfrag}
\graphicspath{{./Figures/}}
\usepackage{url}
\usepackage{hyperref}

\newtheorem{prop}{Proposition}[section]
\newtheorem{theorem}{Theorem}[section]
\newtheorem{cor}{Corollary}[section]
\newtheorem{asm}{Assumption}[section]
\newtheorem{defi}{Definition}[section]
\newtheorem{lemma}{Lemma}[section]
\newtheorem{remark}{Remark}[section]

\numberwithin{equation}{section}

\makeatletter \@addtoreset{equation}{section}

\def\L{\mathcal{L}}
\def\E{\mathbb{E}}
\def\P{\mathbb{P}}

\def\real{\mathbb{R}}

\def\D{\mathbb{D}}

\def\W{\mathbb{W}}

\begin{document}

\allowdisplaybreaks


\centering \Large Stochastic regularization effects of semi-martingales on random functions
\vspace{2em}

\large
\begin{tabular}{c c c}
&&\\
Romain Duboscq \footnotemark[1] & & Anthony R\'eveillac \footnotemark[2] \vspace{1em} \normalsize\\
\end{tabular}

\begin{center}
INSA de Toulouse \footnotemark[3] \\
IMT UMR CNRS 5219 \\
Universit\'e de Toulouse 
\end{center}

\footnotetext[1]{\texttt{romain.duboscq@insa-toulouse.fr}}
\footnotetext[2]{\texttt{anthony.reveillac@insa-toulouse.fr}}
\footnotetext[3]{135 avenue de Rangueil 31077 Toulouse Cedex 4 France}

\abstract{\noindent 
In this paper we address an open question formulated in \cite{flandoli2010well}. That is, we extend the It\^o-Tanaka trick, which links the time-average of a deterministic function $f$ depending on a stochastic process X and $F$ the solution of the Fokker-Planck equation associated to X, to random mappings $f$. To this end we provide new results on a class of adpated and non-adapted Fokker-Planck SPDEs and BSPDEs. }

\vspace{1em}

\vspace{1em}

\renewcommand{\thefootnote}{\fnsymbol{footnote}}

\section{Introduction}

In \cite{flandoli2010well}, the authors analyzed the effects of a multiplicative stochastic perturbation on the well-posedness of a linear transport equation. One of the key tool in their analysis is the so-called \textit{Itô-Tanaka trick} which links the time-average of a function $f$ depending on a stochastic process and $F$ the solution of the Fokker-Planck equation associated to the stochastic process. More precisely, the formula reads as
\begin{equation}\label{eq:ItoTanakaTrick}
\int_0^T f(t,X_t^x) dt = -F(0,x) - \int_0^T \nabla F(t,X_t^x)\cdot dW_t, \; \P-a.s.
\end{equation}
where $(X^x_t)_{t\geq 0}$ is a solution of the stochastic differential equation
\begin{equation}\label{eq:ItoTanakaSDE}
X_t^x = x +\int_0^t  b(s,X_s^x) ds + W_t,
\end{equation}
and $F$ is the solution of the backward Fokker-Planck equation
\begin{equation}\label{eq:ItoTanakaFP}
F(t,x) = -\int_t^T \left(\frac 1 2 \Delta + b(s,x)\cdot \nabla \right) F(s,x)ds - \int_t^T f(s,x) ds.
\end{equation}
In \cite{krylov2005strong}, by means of suitable regularity results for solutions of parabolic equations in $L^q\hspace{-0.3em}-\hspace{-0.3em}L^p$ spaces, the authors showed, assuming $f,b\in E : = L^q([0,T];L^p(\mathbb{R}^d))$ with $2/q+d/p<1$, that $F\in L^q([0,T];W^{2,p}(\mathbb{R}^d))$. Hence, in the weak sense, $F$ has $2$ additional degrees of regularity compared to $f$ in $E$. Thus, formula \eqref{eq:ItoTanakaTrick} tells us that the time-average of $f$ with respect to the stochastic process $(X_t^x)_{t\geq 0}$ is more regular than $f$ itself (it has $1$ additional degree of regularity). This is what we call \textit{a stochastic regularization effect} or \textit{regularization by noise}. In this paper, we investigate the following open question stated in \cite{flandoli2010well}:\vspace{0.5em}

\noindent\textit{"The generalization to nonlinear transport equations, where $b$ depends on $u$ itself, would be a major next step for applications to fluid dynamics but it turns out to be a difficult problem. Specifically there are already some difficulties in dealing with a vector field $b$ which depends itself on the random perturbation $W$. There is no obvious extension of the Itô-Tanaka trick to integrals of the form $\int_0^T f(\omega,s,X^x_s(\omega))ds$ with random $f$."}\\\vspace{0.5em}

\noindent
A major "pathology" in this problem is that there are simple examples of random functions $f$ for which the Itô-Tanaka trick does not work anymore. As an example, consider a random function $\tilde{f}$ of the form
\begin{equation*}
\tilde{f}(\omega,s,x) : = f(x-W_s(\omega)),
\end{equation*}
where $(W_t)_{t\geq 0}$ is the Brownian motion from \eqref{eq:ItoTanakaSDE}.
This gives, for $b=0$ in \eqref{eq:ItoTanakaSDE},
\begin{equation*}
\int_0^T \tilde{f}(\omega,t,W_t+x) dt = \int_0^T f(t,x) dt,
\end{equation*}
which does not bring any additional regularity.\\\vspace{0.5em} 

\noindent
It turns out that, when $f$ is a random function, the solution $F$ to \eqref{eq:ItoTanakaFP} is not adapted anymore to $\left(\mathcal{F}_t^W\right)_{t\in[0,T]}$ the filtration of the Brownian motion, making the stochastic integral on the right-hand side of \eqref{eq:ItoTanakaTrick} ill-posed.\\\vspace{0.5em} 

\noindent
In this paper we tackle this difficulty by considering another equation which is the \textit{adapted} version of the Fokker-Planck equation \eqref{eq:ItoTanakaFP}. More precisely, we show in Theorem \ref{thm:Main} that given a random function $f$ which depends in an adapted way on a standard Brownian motion $(W_t)_{t\geq 0}$, the following formula holds
\begin{equation}\label{eq:IntroItoWentzellTanaka}
\int_0^T f(t,X_t^x) dt = -F(0,x) - \int_0^T \left(\nabla F(s,X_s^x) + Z(s,X_s^x)\right)dW_s - \int_0^T \nabla Z(s,X_s^x) ds, \; \P-a.s.
\end{equation}
where $(F,Z)$ is the \textit{predictable} solution of the following backward stochastic partial differential equation (BSPDE)
\begin{equation*}
F(t,x) = -\int_t^T \left(\frac 1 2 \Delta + b(s,W_{(s)},x)\cdot \nabla \right) F(s,x)ds - \int_t^T f(s,x) ds - \int_t^T Z(s,x) dW_s, 
\end{equation*}
and $(X^x_t)_{t\geq 0}$ is a weak solution of the stochastic differential equation
\begin{equation*}
X_t^x = x +\int_0^t  b(s,W_{(s)},X_s^x) ds + W_t.
\end{equation*}
We name \eqref{eq:IntroItoWentzellTanaka} the \textit{Itô-Wentzell-Tanaka} trick as the derivation of \ref{eq:IntroItoWentzellTanaka} call for the use of the It\^o-Wentzell formula in place of the classical It\^o formula which allows one to give a semimartingale type decomposition of $F(t,X_t^x)$ when $F(t,x)$ is itself a semimartingale random field. Note that we also allow $b$ to depend on the Brownian motion $W$. This contrasts with the classical It\^o-Tanaka trick where both $f$ and $b$ must be deterministic mappings. The derivation of this formula calls for a study of the Fokker-Planck BSPDE. In this direction, incidentally we prove new results as Theorem \ref{prop:ExistAdapt} on this equation in particular by allowing only $L^q\hspace{-0.3em}-\hspace{-0.3em}L^p$ regularity on its coefficients together with a representation of its solution in terms of the solution to the non-adapted SPDE and of its Malliavin derivative by providing a methodology which generalizes: the well-known \textit{linearization technique} used for linear BSDEs and deterministic semigroups (see \cite[Proposition 2.2]{ElKaroui_Peng_Quenez}), and a Feynman-Kac formula for BSPDEs related to Forward-Backward SDEs as in \cite[Corollary 6.2]{Ma_Yong_97}. We also prove that the solution processes $(Y,Z)$ to the equation are Malliavin differentiable. The study of the BSPDE relies on the one of the non-adapted Fokker-Planck equation in Section \ref{section:non-adaptedFK}. \\\vspace{0.5em}

\noindent
There are well-known results concerning the regularization effects of stochastic process on \textit{deterministic} functions (see the survey of Flandoli \cite{flandoli2011random}) but, to our knowledge, there exists no similar results in the case of random functions. The phenomenon is widely used in the  recovery of the strong uniqueness of solutions of stochastic differential equations  (SDE) with singular drifts \cite{davie2007uniqueness,gyongy2001stochastic,meyer2010construction,krylov2005strong,veretennikov1981strong,zvonkin1974transformation}. It has been generalized to SDE in infinite dimension \cite{da2010pathwise,da2013strong,mohammed2015sobolev} and the conditions for the existence of a stochastic flow has also drawn attention \cite{attanasio2010stochastic,flandoli2010flow,zhang2010stochastic}. Another direction of interest is the improvement of the well-posedness of stochastic partial differential equations (SPDE). In particular, the stochastically perturbed linear transport equation has received a lot of interest \cite{attanasio2009zero,catellier2015rough,fedrizzi2013noise,flandoli2010well}. More recent works provide extensions to nonlinear SPDE, see for instance \cite{barbato2010uniqueness,flandoli2011full,gubinelli2013regularization} for models from fluid mechanics and \cite{chouk2013nonlinear,chouk2014nonlinear,debussche20111d} for dispersive equations. Let us also mention that the type of processes that yield a regularization effect is not restricted to semi-martingales. For instance, in \cite{priola2012pathwise,zhang2013stochastic} where $\alpha$-stable processes have been considered and, in \cite{catellier2012averaging}, where the authors showed a regularization phenomenon using rough paths (in particular for the fractional Brownian motion).\\ \vspace{0.5em}

\noindent
The paper is organized as follows. In Section \ref{section:notations} we make precise the definitions and the notations that will be used later on. This includes some material on Malliavin calculus especially for random fields. Then, in Section \ref{section:adandnonadFK} we introduce the transport SDE under interest and we study the adapted and the non-adapted Fokker-Planck equations. The It\^o-Tanaka-Wentzell trick, together with an example, is presented in Section \ref{section:ITK}. 

\section{Notations and preliminaries}
\label{section:notations}

\subsection{Main notations}

Throughout this paper $T$ will be a fixed positive real number and $d$ denotes a fixed positive integer. For any $x$ in $\real^d$, we denote by $|x|$ the Euclidian norm of $x$. Let $(E,\|\cdot\|_E)$ be a Banach space, we set $\mathcal{B}(E)$ the Borelian $\sigma$-field on $E$. For given Banach spaces $E, F$ and any $p\geq 0$, we set $L^p(E;F)$ the set of $\mathcal{B}(E)\backslash\mathcal{B}(F)$-measurable mappings $f:E\to F$ such that 
$$ \|f\|_{L^p(E;F)}^p:=\int \|f(x)\|_F^p \mu(dx) <+\infty, $$
where $\mu$ is a non-negative measure on $(E,\mathcal{B}(E))$ (the Borelian $\sigma$-field on $E$). Naturally the norm depends on the choice of $\mu$ that will be made explicit in the context. If $F=\real^n$, $n\in\mathbb{N}$, then we simply set $L^p(E):=L^p(E;\real^n)$.  We also denote by $\mathcal{C}^0(E)$ (resp. $\mathcal{C}^0_b(E)$) the set of continuous (resp. bounded continuous) real-valued mappings $f$ on $E$. 
\\\vspace{0.5em}
\noindent
For any $p>1$ we set $\bar{p}$ the H\"older conjugate of $p$. 
\\\vspace{0.5em}
\noindent
For any mapping $\varphi:\real^d \to \real$ we denote by $\frac{\partial \varphi}{\partial x_i}$ the $i$-th partial derivative of $\varphi$ ($i=1,\cdots,n$), by $\nabla \varphi:=(\frac{\partial \varphi}{\partial x_1},\ldots,\frac{\partial \varphi}{\partial x_d})$ the gradient of $\varphi$ (when it is well-defined), and by $\Delta \varphi$ its Laplacian. For a multi index $k:=(k_1,\cdots,k_d)$ in $\mathbb{N}^d$, we set $\nabla^k \varphi:=\frac{\partial^{k_1+\cdots+k_d}\varphi}{\partial x_{1} \ldots \partial^{k_d}} \varphi$ and $|k|:=\sum_{i=1}^d k_i$.\\\vspace{0.5em}
 
\noindent
For $p,m\in\mathbb{R}^+$, we set 
\begin{equation*}
W^{m,p}(\mathbb{R}^d) = \left\{ \varphi\in L^p(\mathbb{R}^d);  \mathcal{F}^{-1}\left(([1+|\xi|^2]^{m/2}\hat{\varphi}\right)\in L^p(\mathbb{R}^d)) \right\},
\end{equation*}
the usual Sobolev spaces equipped with its natural norm
\begin{equation*}
\|\varphi\|_{W^{m,p}(\mathbb{R}^d)}:= \left\|\mathcal{F}^{-1}\left(([1+|\xi|^2]^{m/2}\hat{\varphi}\right)\right\|_{L^p(\mathbb{R}^d))},
\end{equation*}
where $\hat{\varphi}(\xi) = \mathcal{F}(\varphi)(\xi)$ and $\mathcal{F}$ (resp. $\mathcal{F}^{-1}$) denotes the Fourier transform (resp. the inverse Fourier transform).
Let $n,k\in\mathbb{N}$ and $\alpha\in(0,1)$. We denote by $\mathcal{C}^{k,\alpha}_b(E)$, the set of bounded functions having bounded derivatives up to order $k$ and with $\alpha$-Hölder continuous $k$th partial derivatives. It is equipped with the norm
\begin{equation*}
\|\varphi\|_{\mathcal{C}_b^{k,\alpha}(E)} : =  \|\varphi\|_{\mathcal{C}_b^k(E)} +  \sup_{|\ell| = k}\sup_{x\neq y}\frac{|\nabla^{\ell}f(x)-\nabla^{\ell}f(y)|}{|x-y|^{\alpha}},
\end{equation*}
where $\|\varphi\|_{\mathcal{C}_b^k(E)} : =\sum_{|\ell|\leq k}\sup_{x\in E}|\nabla^{\ell}f(x)|$.
Finally $\mathcal{C}_0^\infty(\real^{n})$, ($n\in \mathbb{N}^\ast$) stands for the set of infinitely continuously differentiable function with compact support.\\\vspace{0.5em}
\noindent
Throughout this paper $C$ will denote a non-negative constant which may differ from line to line. 

\noindent Unless stated otherwise, we always assume that the real numbers $p,q\in (2,\infty)$ verify
\begin{equation*}
\frac d p + \frac 2 q <1.
\end{equation*}

\subsection{Malliavin-Sobolev spaces}

In this section we recall the classical definitions of Malliavin-Sobolev spaces presented in \cite{Nualartbook} and extended them to functional valued random variables that from now on we will refer as random fields. We start with some facts about Malliavin's calculus for random variables.

\subsubsection{Malliavin calculus for random variables}
\label{section:Malliavinclassical}

Let $(\Omega,\mathcal{F},\P)$ be a probability space and $W:=(W_t)_{t\in [0,T]}$ a Brownian motion on this space (to the price of heavier notations all the definitions and properties in this section and of the next one extend to a $d$-dimensional Brownian motion). We assume that $\mathcal{F} = \sigma\left(W_t, \; t\in [0,T]\right)$.\\\vspace{0.5em}

\noindent
Let $\mathcal{S}^{rv}$ be the set of cylindrical functionals, that is the set of random variable $\beta$ of the form: $$\beta=\varphi(W_{t_1},\cdots,W_{t_n})$$ with $\mathbb{N}^\ast$, $\varphi:\real^n\to \real$ in $\mathcal{C}_0^\infty(\real^n)$ and $0 \leq t_1< \cdots < t_n\leq T$. For an element $\beta$ in $\mathcal{S}^{rv}$, we set $DF$ the $L^{2}([0,T])$-valued random variable as:
$$ D_\theta \beta:=\sum_{i=1}^n \frac{\partial \varphi}{\partial x_i}(W_{t_1},\cdots,W_{t_n}) \textbf{1}_{[0,t_i]}(\theta), \quad \theta \in [0,T].$$
For a positive integer $p\geq 1$, we set $\D^{1,p}$ the closure of $\mathcal{S}^{rv}$ with respect to the norm:
$$\|\beta\|_{\D^{1,p}}^p:=\E[|\beta|^p] + \E\left[\left(\int_0^T |D_\theta \beta|^2 d\theta\right)^{p/2}\right].$$
To $D$ is associated a dual operator denoted $\delta$ defined through the following integration by parts formula:
\begin{equation}
\label{eq:DIPP}
\E[\beta \delta(u)] = \E\left[\int_0^T D_t \beta \; u_t dt\right],
\end{equation}
for any $\beta$ in $\D^{1,2}$ and any $L^2([0,T])$-valued random variable $u$ such that there exists a positive constant $C$ such that  
$\left|\E\left[\int_0^T D_t \gamma u_t dt \right]\right| \leq C \|\gamma\|_{\D^{1,2}}, \; \forall \gamma \in \D^{1,2}$. In particular if $u:=(u_t)_{t\in [0,T]}$ is a predictable process then $\delta(u)=\int_0^T u_t dW_t$. In addition, according to \cite[Proposition 1.3.4]{Nualartbook}, for any $\beta$ in $\mathcal{S}$ and any $h$ in $L^p([0,T])$ (with $p\geq 2$), $\delta(h \beta)$ is well-defined and satisfies 
\begin{equation}
\label{eq:deltaprod}
\delta(h \beta) = \beta \delta(h) - \int_0^T h_t D_t \beta dt.
\end{equation} 

\subsubsection{Malliavin calculus for random fields}

We now extend these definitions to random fields that is to measurable mappings $F:\Omega \times \real^d \to \real$. More precisely, we consider $\mathcal{S}$ be the set of cylindrical fields, that is the set of random fields $F$ of the form: $$F=\varphi(W_{t_1},\cdots,W_{t_n},x)$$ with $\varphi:\real^n\times\real^d \to \real$ in $\mathcal{C}_0^\infty(\real^{n+d})$. We fix $p$ an integer with $p\geq 2$. For an element $F$ in $\mathcal{S}$, we set $DF$ the $L^{p}([0,T])$-valued random field as:
$$ D_\theta F:=\sum_{i=1}^n \frac{\partial \varphi}{\partial x_i}(W_{t_1},\cdots,W_{t_n},x) \textbf{1}_{[0,t_i]}(\theta), \quad \theta \in [0,T].$$

\noindent
Note that for $F$ in $\mathcal{S}$, $D \nabla^k F = \nabla^k D F $ for any multi index $k$. In addition, an integration by parts formula for the operators $D \nabla^k$ can be derived as follows.

\begin{lemma}\label{lem:MalliavinDerivativ1}
Let $F$ in $\mathcal{S}$, $h$ in $L^p([0,T])$ and $G$ in $\mathcal{S}$. Let $k$ be a multi-index in $\mathbb{N}^d$, then the following integration by parts formula holds true: 
\begin{equation}
\label{eq:IPPDgrad}
\E\left[\int_0^T \int_{\real^d} D_t \nabla^k F(x) h_t G(x) dx dt\right] = \E\left[\int_{\real^d} F(x) \; \delta((\nabla^k)^\ast G(x) h)dx \right],
\end{equation}
where $(\nabla^k)^\ast$ denotes the dual operator of $\nabla^k$.
\end{lemma}

\begin{proof}
By the Malliavin-integration by parts formula (see \textit{e.g.} \cite[Lemma 1.2.1]{Nualartbook}) and by the classical integration by parts formula in $\real^d$ we have that:
\begin{align*}
\E\left[\int_0^T \right.&\hspace{-0.1cm}\left.\int_{\real^d} D_t \nabla^k F(x) h_t G(x) dx dt\right] \\
&=\int_{\real^d} \E\left[\int_0^T D_t \nabla^k F(x) h_t G(x) dt\right]dx \\
&=\int_{\real^d} \E\left[\nabla^k F(x) \delta(G(x) h) \right]dx, \textrm{ by } \eqref{eq:DIPP} \\
&=\int_{\real^d} \E\left[\nabla^k F(x) G(x) \delta(h) \right]dx - \int_{\real^d} \E\left[\nabla^k F(x) \int_0^T D_t G(x) h_t dt \right]dx, \textrm{ by } \eqref{eq:deltaprod} \\
&=\E\left[\int_{\real^d} F(x) (\nabla^k)^\ast G(x) dx \delta(h) \right] - \E\left[\int_{\real^d} F(x) (\nabla^k)^\ast \int_0^T D_t G(x) h_t dt dx\right] \\
&=\E\left[\int_{\real^d} F(x) \left((\nabla^k)^\ast G(x) \delta(h)-\int_0^T D_t (\nabla^k)^\ast G(x) h_t dt\right)dx \right]\\
&=\E\left[\int_{\real^d} F(x) \; \delta((\nabla^k)^\ast G(x) h)dx \right], \textrm{ by } \eqref{eq:deltaprod}.
\end{align*}
\end{proof}

This integration by parts formula allows us to prove that the operators $D \nabla^k$ are closable.

\begin{lemma}\label{lem:MalliavinDerivativ2}
Let $p \geq 2$ and $k$ be in $\mathbb{N}^d$. The operators $D \nabla^k $ (and so $\nabla^k D$) are closable from $\mathcal{S}$ to $L^p(\Omega \times \real^d; L^p([0,T]))$.
\end{lemma}

\begin{proof}
Let $(F_n) \subset \mathcal{S}$ a sequence of random fields which converges in $L^p(\Omega \times \real^d;L^p(\mathbb{R}^d))$ to $0$ and such that $(D \nabla^k F_n)_n$ converges in $L^p(\Omega\times \real^d; L^p([0,T]))$ to some element $\eta$ in $L^p(\Omega\times \real^d; L^p([0,T]))$. Let $h$ in $L^p([0,T])$ and $G:\real^d \to \real$ in $\mathcal{S}$. We recall that $\bar{p}:=\frac{p}{p-1}$. For any $n\geq 1$, it holds that

\begin{align*}
&\E\left[\int_{\real^d} \int_0^T \eta(t,x) h_t dt G(x) dx\right] \\
&= \E\left[\int_{\real^d} \int_0^T (\eta(t,x)-D_t \nabla^k F^n(x)) h_t dt G(x) dx\right] + \E\left[\int_{\real^d} \int_0^T D_t \nabla^k F^n(x) h_t dt G(x) dx\right]\\
&= \E\left[\int_{\real^d} \int_0^T (\eta(t,x)-D_t \nabla^k F^n(x)) h_t dt G(x) dx\right] + \E\left[\int_{\real^d} F^n(x) \; \delta((\nabla^k)^\ast G(x) h)dx\right],
\end{align*}
where we have used the integration by parts formula \eqref{eq:IPPDgrad}. 
We estimate the two terms above separately. For the first one, using successive H\"older's Inequality, we have that
\begin{align*}
&\left|\E\left[\int_{\real^d} \int_0^T (\eta(t,x)-D_t \nabla^k F^n(x)) h_t dt G(x) dx\right]\right|\\
&\leq \E\left[\int_{\real^d} \int_0^T |\eta(t,x)-D_t \nabla^k F^n(x)|^p dt dx\right]^{1/p} \E[\|G\|_{L^{\bar{p}}(\real^d)}^{\bar p}]^{1/{\bar{p}}} \|h\|_{L^{\bar{p}}([0,T])}\\
&\underset{n\to +\infty}{\longrightarrow} 0.
\end{align*}
The second term can be estimated as follows (using also H\"older's inequality and \eqref{eq:deltaprod}).
\begin{align*}
&\left|\E\left[\int_{\real^d} F^n(x) \; \delta((\nabla^k)^\ast G(x) h)dx\right]\right|\\
&=\left|\E\left[\int_{\real^d} F^n(x) (\nabla^k)^\ast G(x) \delta(h)dx\right] - \E\left[\int_{\real^d} F^n(x) \int_0^T D_t (\nabla^k)^\ast G(x) h_t dt dx\right]\right| \\
&\leq C \E\left[\int_{\real^d} |F^n(x)|^p dx\right]^{1/p} \left(\E\left[\delta(h)^{2 \bar{p}}\right]^{1/(2 \bar{p})} \vee \|h\|_{L^2([0,T])}\right)\\
&\times \left(  \E\left[\|(\nabla^k)^\ast G\|_{L^{2\bar{p}}(\real^d)}^{2 \bar p}\right]^{1/(2 \bar{p})}+ \E\left[\int_{\real^d} \left(\int_0^T |D_t (\nabla^k)^\ast G(x)|^2 dt\right)^{\frac{\bar{p}}{2}} dx\right]^{1/\bar{p}}\right)  \\ 
&\underset{n\to +\infty}{\longrightarrow}0.
\end{align*}
Combining the previous estimates and relations we conclude that
$$ \E\left[\int_{\real^d} \int_0^T \eta(t,x) h_t dt G(x) dx\right]=0.$$
The conclusion follows from the fact that the set of elements of the form $G h$ with $h$ in $L^p([0,T])$ and $G$ in $\mathcal{S}$ is dense in $L^p(\Omega\times \real^d;L^p([0,T]))$.
\end{proof}

\begin{remark}
Lemma \ref{lem:MalliavinDerivativ1} and Lemma \ref{lem:MalliavinDerivativ2} still holds if we replace the differential operator $\nabla^k$ with the Bessel potential $(1-\Delta)^{m/2}$ for any $m\in\mathbb{R}^+$.
\end{remark}

\noindent
For a positive integer $m$, we set $\D^{1,m,p}$ the closure of $\mathcal{S}$ with respect to the norm:
\begin{equation}\label{eq:NormD}
 \|F\|_{\D^{1,m,p}}^p:= \|F\|_{\W^{m,p}}^p + \int_0^T \E\left[\|D_\theta F\|_{W^{m,p}(\mathbb{R}^d)}^p\right] d\theta,
 \end{equation}
 where we denote
 \begin{equation*}
 \|F\|_{\W^{m,p}}^p:= \E\left[\|F\|_{W^{m,p}(\mathbb{R}^d)}^p\right].
 \end{equation*}
In addition, for $F$ in $\D^{1,m,p}$, we have since $p\geq 2$:
$$ \|F\|_{\D^{1,m,p}}^p \geq \int_{\real^d} \|F\|_{\D^{1,p}}^p dx + \sum_{|k|\leq m} \int_{\real^d} \|\nabla^k F\|_{\D^{1,p}}^p dx, $$
with equality if $p=2$. 
\begin{remark}\label{rem:MalliavinIncl}
In particular, if a random field $F$ belongs to $\D^{1,m,p}$, then for a.e. $(t,x)$, $\omega\mapsto \nabla^k F(t,x)(\omega)$ belongs to the classical Malliavin space $\D^{1,p}$ whose definition has been recalled in Section \ref{section:Malliavinclassical} (for any $k$ such that $|k|\leq m$) for random variables that depend only on $\omega$ and not on $(t,x)$.
\end{remark}
\noindent
We conclude this section on the Malliavin derivative by introducing the space $\D_q^{1,m,p}:=L^q([0,T];\D^{1,m,p})$ (with $p,q\geq 2$) which consists of mappings $F:[0,T] \times \Omega \times \real^d \to \real$ such that
$$ \|F\|_{\D_q^{1,m,p}}^q:=\int_0^T \|F(t,\cdot)\|_{\D^{1,m,p}}^q dt<+\infty.$$ 
Furthermore, we extend the definition $\W^{m,p}$-norm accordingly
\begin{equation*}
\|F\|_{\W^{m,p}_q}^q =  \int_0^T\|F(t,\cdot)\|_{\W^{m,p}}^q dt. 
\end{equation*}

\section{Fokker-Planck SPDEs and BSPDEs}
\label{section:adandnonadFK}

In this section, we study the transport SDE of interest together with two related Fokker-Planck equations. The first one which will be considered in Section \ref{section:non-adaptedFK} and will be referred as the \textit{non-adapted} (or SPDE) Fokker-Planck equation. The second one will be called the \textit{adapted} (or BSPDE) Fokker-Planck equation associated to the SDE, and will be introduced and studied in Section \ref{section:adaptedFK}. This equation will be fundamental to derive in Section \ref{section:ITK} the stochastic counterpart of the It\^o-Tanaka trick (that we will name then It\^o-Tanaka-Wentzell trick).  

\subsection{A SDE with random drift}

In analogy to \cite{flandoli2010well,krylov2005strong,meyer2010construction}, we consider the following SDE:
\begin{equation}
\label{eq:SDE}
X_t = X_0 +\int_0^t  b(s,W_{(s)},X_s) ds + W_t, \quad t\in [0,T],
\end{equation}
where $b$ is assumed to be a $\mathcal{B}([0,T]\times \mathcal{C}([0,T])\times\real^d)$-measurable map, $X_0$ is in $\mathbb{R}^d$ and $W$ is a $d$-dimensional Brownian motion. To begin with, let us recall the definition of a weak solution to Equation (\ref{eq:SDE}).
\begin{defi}
\label{definition:weaksol}
A weak solution is a triple $(X,W)$, $(\Omega,\mathcal{F},\P)$, $(\mathcal{F}_t)_{t\in[0,T]}$ where
\begin{itemize}
\item $(\Omega,\mathcal{F},\P)$ is a probability space equipped with some filtration $(\mathcal{F}_t)_{t\in[0,T]}$ that satisfies the usual conditions,
\item $X$ is a continuous, $(\mathcal{F}_t)_{t\in[0,T]}$-adapted $\mathbb{R}^d$-valued process, $W$ is a $d$-dimensional $(\mathcal{F}_t)_{t\in[0,T]}$-Wiener process on the probability space,
\item $\P(X(0) = X_0) = 1$ and $\P(\int_0^t |b(s,W_{(s)},X_s)|ds <+\infty) =1$, $\forall t\in [0,T]$,
\item Equation (\ref{eq:SDE}) holds for all $t$ in $[0,T]$ with probability one.
\end{itemize}
\end{defi}

\begin{asm}
There exists a weak solution $(X,W)$, $(\Omega,\mathcal{F},\P)$, $(\mathcal{F}_t)_{t\in[0,T]}$ to the SDE (\ref{eq:SDE}).
\end{asm}

By defintion, $W$ is a $(\mathcal{F}_t)_{t\in[0,T]}$-Brownian motion. So we denote by $(\mathcal{F}_t^W)_{t\in[0,T]}$ its natural completed right-continuous filtration which satisfies $\mathcal{F}_t^W \subset \mathcal{F}_t$ for any $t\in [0,T]$. In the following, the spaces $\D^{r,m,p}$ or $\D_q^{r,m,p}$ are understood to be defined with respect to $(\Omega,\mathcal{F}_T^W,\P)$.

We now give a simple proof of existence and uniqueness of a weak solution to \eqref{eq:SDE} under some non-optimal assumptions.

\begin{prop}
Let $b\in\mathcal{C}^1_b(\mathbb{R}^d;L^q([0,T];L^p(\mathbb{R}^d))$. Then there exists a unique weak solution to the SDE
\begin{equation}\label{eq:NoOptimSDE}
X_t = X_0 + \int_0^t b(s,W_s,X_s)ds + W_t, \quad t\in [0,T].
\end{equation}
\end{prop}

\begin{proof}
The proof is based on the Girsanov's theorem. Let us first remark that $L(t,x) : = \sup_{y\in\mathbb{R}^d} |\nabla_y b(t,y,x)|$ and $\tilde{b}(t,x) : =\sup_{y\in\mathbb{R}^d}|b(t,y,x)|$ belong in $L^q([0,T];L^p(\mathbb{R}^d))$. Thus, since $2/q+d/p<1$, by \cite[Lemma 3.2]{krylov2005strong} we have, $\forall \kappa\in\mathbb{R}^+$ and $k=1,2$,
\begin{equation}\label{eq:BoundAllMoments}
\E\left[e^{\kappa\int_0^T L(s,W_s)^kds}\right] + \E\left[e^{\kappa\int_0^T \tilde{b}(s,W_s)^kds}\right]< +\infty,
\end{equation}
where $W$ is a standard Brownian motion.

\noindent Let $(X_t)_{t\geq 0}$ a standard Brownian motion on a  probability space $(\Omega,\mathcal{F},\P)$ equipped with the filtration $\left(\mathcal{F}_{t}\right)_{t\in[0,T]}$. We consider the following SDE
\begin{equation}\label{eq:NoOptimDualSDE}
Y_t = Y_0 - \int_0^t b(s,Y_t,X_t)ds + X_t, \quad t\in [0,T].
\end{equation}
In this step, we prove that there exists a unique solution to \eqref{eq:NoOptimDualSDE}. Since $b$ is Lipschitz, the uniqueness is obtain by a Gronwall lemma. Moreover, by using classical \textit{a priori} estimates for Lipschitz SDE, we obtain
\begin{equation*}
\mathbb{E}\left[ \sup_{t\in[0,T]} |Y_t|^2 \right] \leq C \left( |Y_0|^2 + T + \E\left[ \int_0^T\left( |b(s,0,X_s)|^2 + L(s,W_s)^2\right)ds\right]\right),
\end{equation*}
which yields the existence of a strong solution.

\noindent By \eqref{eq:BoundAllMoments}, we have, $\forall \kappa\in\mathbb{R}^+$,
\begin{equation*}
\E\left[e^{\kappa\int_0^T|b(s,Y_s,X_s)|^2 ds}\right] \leq \E\left[e^{\kappa\int_0^T \tilde{b}(s,X_s)^2ds}\right] < +\infty.
\end{equation*}
We deduce that
\begin{equation*}
\rho(\cdot) : = e^{\int_0^{\cdot} b(s,Y_s,X_s)ds - \frac 1 2 \int_0^{\cdot} |b(s,Y_s,X_s)|^2ds},
\end{equation*}
is a martingale under $\P$ by Novikov's criterion. Hence, by Girsanov's theorem, the process $Y$ is a Brownian motion under the measure $\mathbb{Q}$ given by $\frac {d\mathbb{Q}}{d\P} = \rho(T)$. Thus, by rewriting $Y$ as $W$, the triple $(X,W),(\Omega,\mathcal{F},\mathbb{Q}),(\mathcal{F}_t)_{t\in[0,T]}$ is a weak solution to the SDE \eqref{eq:NoOptimSDE}.

\end{proof}

\noindent
We will need below several technical results that we present now. In the following, we denote by $(P_{t,s})_{s\geq t\geq 0}$ the heat semigroup.

\begin{lemma}\label{lem:EstimChaleur}
Let $1<q,p<+\infty$. Then, there exists a constant $C$ such that, $\forall f\in \D_q^{1,0,p}$,
\begin{equation}\label{eq:EstimChaleur}
\left\|\int_t^T P_{t,s} f(s,x) ds\right\|_{\D_q^{1,2,p}} \leq C \|f\|_{\D_q^{1,0,p}},
\end{equation}
and, another constant $C_T>0$ such that, $\forall \varepsilon>0$ and $\forall \phi\in\D^{1,2-2/q+\varepsilon,p}$,
\begin{equation}\label{eq:EstimChaleur2}
\|P_{t,T}\phi\|_{\D^{1,2,p}_q} \leq C_T \|\phi\|_{\D^{1,2-2/q+\varepsilon,p}}.
\end{equation}
\end{lemma}

\begin{proof}
\textbf{Second estimate:} The second estimate is a direct consequence of the a similar estimate in deterministic spaces. It is based on the following lemma \cite[Theorem 7.2]{krylov1999analytic}.
\begin{lemma}\label{lem:1Krylov}
Let $v\in L^p(\mathbb{R}^d)$, $\alpha\in[0,1]$, and $t>0$. There exists a constant $C>0$ such that
\begin{equation}
\|e^{-t}P_{0,t}v\|_{W^{2\alpha,p}(\mathbb{R}^d)}\leq \frac C {t^{\alpha}} \|v\|_{L^{p}(\mathbb{R}^d)} \hspace{0.2cm} \mbox{ and } \hspace{0.2cm} \|(P_{0,t}-1)v\|_{L^p(\mathbb{R}^d)} \leq C t^{2\alpha}\|v\|_{W^{2\alpha,p}(\mathbb{R}^d)}.
\end{equation}
\end{lemma}
By setting $\alpha = 1/q-\varepsilon/2$ and $v = (1-\Delta)^{m/2}\phi$, with $m=2-2/q+\varepsilon$, in Lemma \ref{lem:1Krylov}, we obtain
 \begin{equation*}
 \|P_{t,T}\phi\|_{\D^{1,2,p}} \leq \frac{Ce^{T-t}}{(T-t)^{1/q-\varepsilon/2}} \|\phi\|_{\D^{1,2-2/q+\varepsilon},p},
 \end{equation*}
 thus, this yields the desired estimate
 \begin{equation}
  \|P_{t,T}\phi\|_{\D^{1,2,p}_q} \leq C \left(\int_0^T \frac{e^{q\tau}}{\tau^{1-q\varepsilon/2}}d\tau\right)^{1/q} \|\phi\|_{\D^{1,2-2/q+\varepsilon},p}.
 \end{equation}
 
\textbf{First estimate:} For the first estimate, the arguments of the proof are similar to those of \cite[Theorem 1.1]{krylov2001heat}. First, let us remark that in the case $p = q$, Estimate \eqref{eq:EstimChaleur} can be deduce directly by using the classical inequality
\begin{equation*}
\left\|\int_t^T P_{t,s} g(s,x) ds\right\|_{L^p([0,T], W^{2,p})} \leq C \|g\|_{L^p([0,T]\times\real^d)}, \; \forall g \in L^p([0,T]\times\real^d).
\end{equation*}
Therefore, it remains to prove estimate \eqref{eq:EstimChaleur} for $q \neq p$. To do this, we apply the Calder\'on-Zygmund Theorem in Banach spaces (see \cite[Theorem 1.4]{krylov2001heat} for a precise statement). More precisely, we define the operator
\begin{equation*}
\mathcal{A}f(t,x) : = \int_{\mathbb{R}} P_{t,s}f(s,x)\bold{1}_{t\leq s\leq T} ds,
\end{equation*}
which is a bounded operator from $L^p(\real,\D^{1,0,p})$ to $L^p(\real,\D^{1,2,p})$ since Estimate \eqref{eq:EstimChaleur}  is valid for $q=p$. Therefore, to apply the Calder\'on-Zygmund Theorem, we only need to prove the following estimate, $\forall t\neq s$,
\begin{equation}
\|\partial_s^{\ell}P_{t,s}f\|_{\D^{1,2,p}} \leq \frac{C}{(s-t)^2}\|f\|_{\D^{1,2,p}},
\end{equation}
for $\ell = 0,1$, which can be deduced from the classical inequality
\begin{equation}\label{ineq:smoothingspace}
\|\nabla^k P_{t,s}f\|_{L^p(\real^d)} \leq \frac C {(s-t)^{|k|/2}}\|f\|_{L^p(\real^d)},
\end{equation}
and the fact that $\partial_s P_{t,s} = \frac 1 2 \Delta P_{t,s}$. This enables us to extend the operator $\mathcal{A}$ to a bounded operator from $L^q(\real,\D^{1,0,p})$ to $L^q(\real,\D^{1,2,p})$, $\forall q\in (1,p]$. Finally, we remark that the adjoint operator of $\mathcal{A}$ is given by
\begin{equation*}
\mathcal{A}^* f(t,x) = \int_0^t P_{s,t} f(s,x) ds.
\end{equation*}
Thus, we are able to apply the same results to $\mathcal{A}^*$ and conclude that $\mathcal{A}$ is also a bounded operator from $L^{\bar{q}}(\real,\D^{1,0,\bar p})$ to $L^{\bar q}(\real,\D^{1,2,\bar p})$, $\forall q\in (1,p]$. This extends the range of $q$ to $(1,\infty)$ for $\mathcal{A}$.
\end{proof}

The next result gives a Schauder estimate on the solution of a backward heat equation with a source term in $\D_q^{1,0,p}$. Its proof is similar to the one from \cite[Theorem 7.2]{krylov1999analytic} and the arguments can be directly extended to the norms $\D^{1,m,p}_q$.

\begin{prop}\label{prop:Holder}
Let $1<q,p < +\infty$, $2/q<\beta\leq 2$ and $f\in \D_q^{1,0,p}$. Denote, for $(t,x)\in[0,T]\times\real^d$,
\begin{equation*}
u(t,x) : = -\int_t^T P_{t,s} f(s,x) ds.
\end{equation*}
Then, there exists a constant $C>0$ independent of $T$ such that, for any $0\leq s\leq t\leq T$,
\begin{equation}
\|u(t)-u(s)\|_{\D^{1,2-\beta,p}} \leq C (t-s)^{\beta/2-1/q} \|f\|_{\D_q^{1,0,p}},
\end{equation}
and, thus,
\begin{equation}
\|u\|_{\mathcal{C}^{0,\beta/2-1/q}_b([0,T]; \D^{1,2-\beta,p})} \leq C\|f\|_{\D_q^{1,0,p}}.
\end{equation}
\end{prop}

A direct consequence of the previous result is the following
\begin{cor}\label{cor:Holder}
Let $f\in \D_q^{1,0,p}$. Denote, for $(t,x)\in[0,T]\times\real^d$,
\begin{equation*}
u(t,x) : = -\int_t^T P_{t,s} f(s,x) ds.
\end{equation*}
Then, for any $\varepsilon\in(0,1)$ satisfying
\begin{equation*}
\varepsilon + \frac d p + \frac 2 q <1,
\end{equation*}
there exists a constant $C>0$ and $\tilde{\varepsilon}>0$ such that, $\forall t\in[0,T]$,
\begin{equation}
\left(E\left[\|u(t,\cdot)\|_{C^{1,\varepsilon}_b (\mathbb{R}^d)}^p\right]+ \E\left[\int_0^T \|D_{\theta} u(t,\cdot)\|_{C^{1,\varepsilon}_b(\mathbb{R}^d)}^p d\theta\right]\right)^{1/p} \leq C (T-t)^{\tilde{\varepsilon}/2}\|f\|_{\D_q^{1,0,p}}.
\end{equation}
\end{cor}

\begin{proof}
Let $\beta = \tilde{\varepsilon} + 2/q$ where $0<\tilde{\varepsilon}<1-(\varepsilon+d/p+2/q)$. The result follows by the Sobolev embedding $\mathcal{C}^{1,\alpha}_b\subset W^{2-\beta,p}$, with $\alpha = 1 - \beta - d/p = 1 - \tilde{\varepsilon} - q/2-d/p>\varepsilon$, and Proposition \ref{prop:Holder}.
\end{proof}

\subsection{The non-adapted Fokker-Planck equation}
\label{section:non-adaptedFK}

We set the linear operator $\L^X_{t}$ on $\mathcal{C}^{\infty}_0(\mathbb{R}^d)$:
\begin{equation*}
\L^X_{t} \varphi(x):= \frac12 \Delta \varphi(x) + b(t,x) \cdot \nabla \varphi(x),
\end{equation*}
and consider here the non-adapted Fokker-Planck equation
\begin{equation}\label{eq:FK}
F(t,x) = \phi(x) - \int_t^T \L^X_{r} F(r,x) dr - \int_t^T f(r,x) dr.
\end{equation}
\begin{defi}
A strong solution to Equation \eqref{eq:FK} is a function $F$ in $\D^{1,2,p}_q$ such that, for all $t\in[0,T]$, we have
\begin{equation}\label{eq:WeakFK}
F(t,x) = \phi(x) - \int_t^T \L^X_{r} F(r,x)dr - \int_t^Tf(r,x) dr.
\end{equation}
\end{defi}

\begin{remark}
\label{rem:non-predictable}
Note that each random variable $F(t,\cdot)$ solution to the previous SPDE is $\mathcal{F}_T$-measurable, and hence it is not adapted (compare with Remark \ref{rem:BSPDE} below). 
\end{remark} 

We provide a Malliavin differentiability analysis for the solution the Fokker-Planck equation (\ref{eq:FK}). We define, $\forall m\geq 0$,
\begin{equation*}
\mathbb{G}^{1,m,p}_q : = \left\{F\in\D_q^{1,m,p}; \partial_t F \in\D_q^{1,0,p}  \right\},
\end{equation*}
and the associated norm
\begin{equation*}
\|F\|_{\mathbb{G}^{1,m,p}_q} : = \|F\|_{\D_q^{1,m,p}} +  \|\partial_t F\|_{\D_q^{1,0,p}}.
\end{equation*}
We begin with a result concerning the existence and uniqueness of a solution to the non-adapted Fokker-Planck equations.

\begin{asm}\label{asm:Cauchy} We assume that there exists a function $\tilde{b}\in L^q([0,T])$ such that, $\forall (t,w)\in[0,T]\times\mathcal{C}([0,T])$,
\begin{equation*}
\|b(t,w,\cdot)\|_{L^p(\mathbb{R}^d)} + \left(\int_0^T \|D_{\theta}b(t,w,\cdot)\|_{L^p(\mathbb{R}^d)}^p d\theta\right)^{1/p} \leq \tilde{b}(t).
\end{equation*}
\end{asm}

\begin{lemma}
Assume that \ref{asm:Cauchy} is in force. Let $u\in\mathbb{G}^{1,2,p}_q$ and denote
\begin{equation*}
\|u(t,\cdot)\|_{\mathbb{H}^{1,p}}^p : = \E\left[ \sup_{x\in\mathbb{R}^d}|\nabla u(t,x)|^p\right]+ \E\left[\int_0^T\sup_{x\in\mathbb{R}^d}|\nabla  D_\theta u(t,x)|^pd\theta\right].
\end{equation*}
The followings estimates hold
\begin{equation}\label{eq:EstimCGCauchy}
\sup_{t\in[0,T]}\|u(t,\cdot)\|_{\mathbb{H}^{1,p}} \leq C_T \|u\|_{\mathbb{G}^{1,2,p}_q},
\end{equation}
where $C_T$ is uniformly bounded with respect to $T$ in compact sets of $\mathbb{R}^+$, and, $\forall t\in[0,T]$,
\begin{align}
\|b(t,\cdot)\cdot \nabla &u(t,\cdot)\|_{\D^{1,0,p}} \leq C\tilde{b}(t)\|u(t,\cdot)\|_{\mathbb{H}^{1,p}}.\label{eq:MalliavinCalculTilde}
\end{align}
\end{lemma}

\begin{proof}
Firstly, let us remark that we have, $\forall u\in \mathbb{G}^{1,2,p}_q$,
\begin{equation}
u(t,x) = -\int_t^T P_{t,r}\left[ \partial_t u(r,x) - \frac 1 2 \Delta u(r,x)\right]dr,\nonumber
\end{equation}
and then, by using Corollary \ref{cor:Holder}, we obtain the estimate
$$\sup_{t\in[0,T]}\|u(t,\cdot)\|_{\mathbb{H}^{1,p}} \leq C_T \|u\|_{\mathbb{G}^{1,2,p}_q}.$$

Secondly, we compute
\begin{align*}
\|b(t,\cdot)\cdot \nabla u(t,\cdot)\|_{\D^{1,0,p}}^p = \hspace{0.1cm}& \E\left[\|b(t,\cdot)\cdot \nabla u(t,\cdot)\|_{L^{p}(\mathbb{R}^d)}^p\right]
\\ & + \int_0^T \E\left[\|D_\theta b(t,\cdot)\cdot \nabla u(t,\cdot) +  b(t,\cdot)\cdot D_\theta\nabla u(t,\cdot) \|_{L^{p}(\mathbb{R}^d)}^p\right] d\theta
\\ \leq \hspace{0.1cm}&\E\left[\|b(t,\cdot)\cdot\nabla u(t,\cdot)\|_{L^{p}(\mathbb{R}^d)}^p\right]
\\ & + C \E\left[\int_0^T\|D_\theta b(t,\cdot)\cdot \nabla u(t,\cdot)\|_{L^{p}(\mathbb{R}^d)}^pd\theta\right] 
\\ & + C \E\left[\int_0^T\| b(t,\cdot)\cdot D_\theta\nabla  u(t,\cdot)\|_{L^{p}(\mathbb{R}^d)}^pd\theta\right].
\end{align*}
Since the Malliavin derivative commutes with the spatial derivative in $L^p$, we obtain
\begin{align*}
\|b(t,\cdot)\cdot \nabla u(t,\cdot)\|_{\D^{1,0,p}}^p \leq \hspace{0.1cm}& \E\left[\|b(t,\cdot)\|_{L^{p}(\mathbb{R}^d)}^p \sup_{x\in\mathbb{R}^d}|\nabla u(t,x)|^p\right]
\\ & + C \E\left[\int_0^T\|D_\theta b(t,\cdot)\|_{L^{p}(\mathbb{R}^d)}^pd\theta \sup_{x\in\mathbb{R}^d}|\nabla u(t,x)|^p\right] 
\\ & + C \E\left[\| b(t,\cdot)\|_{L^{p}(\mathbb{R}^d)}^p \int_0^T\sup_{x\in\mathbb{R}^d}|\nabla  D_\theta u(t,x)|^pd\theta\right].
\end{align*}
Thus, by Assumption \ref{asm:Cauchy}, we have \eqref{eq:MalliavinCalculTilde} as
\begin{align*}
\|b(t,\cdot)\cdot \nabla &u(t,\cdot)\|_{\D^{1,0,p}} \leq\hspace{7cm}\nonumber
\\  &C\tilde{b}(t)\left( \E\left[ \sup_{x\in\mathbb{R}^d}|\nabla u(t,x)|^p\right]+ \E\left[\int_0^T\sup_{x\in\mathbb{R}^d}|\nabla  D_\theta u(t,x)|^pd\theta\right]\right)^{1/p}.
\end{align*}

\end{proof}

\begin{prop}
\label{prop:Malliavindifnonadapted}
Let $f\in \D_q^{1,0,p}$ and $\phi\in\D^{1,2-2/q+\varepsilon,p}$, with $\varepsilon>0$.
Under Assumption \ref{asm:Cauchy}, there exists a unique solution $F$ in $\mathbb{G}_q^{1,2,p}$ to the equation
\begin{equation}\label{eq:ParabolicNadapt}
F(t,x) = P_{t,T}\phi(x)-\int_t^T P_{t,s}f(s,x)ds - \int_t^TP_{t,s}\left[b(s,x)\cdot\nabla F(s,x) \right]ds.
\end{equation}
Moreover, the following estimate on the solution holds
\begin{equation}\label{eq:ParabolicNadaptEsti}
\|F\|_{\mathbb{G}^{1,2,p}_q}\leq  C_{T}\left(\|\phi\|_{\D^{1,2-2/q+\varepsilon,p}}+\|f\|_{\D^{1,0,p}_q}\right),
\end{equation}
where  $C_{T}>0$ depends on $\|\tilde{b}\|_{L^q([0,T])}$ and is uniformly bounded with respect to $T$ on compact sets of $\mathbb{R}^+$.
\end{prop}

\begin{proof}
\textbf{Step 1:}
By using Corollary \ref{cor:Holder} and \eqref{eq:MalliavinCalculTilde}, we have
\begin{align*}
\|F(t,\cdot)\|_{\mathbb{H}^{1,p}}^q \leq& C  \|P_{t,T} \phi\|_{\mathbb{H}^{1,p}}^q + C_{T} \|f\|_{\D_q^{1,0,p}}^q + C_{T} \|b\cdot \nabla F \|_{\D_q^{1,0,p}}^q
\\ \leq& C \| \phi\|_{\mathbb{H}^{1,p}}^q + C_{T} \|f\|_{\D_q^{1,0,p}}^q + C_{T}\int_t^T |\tilde{b}(s)|^q  \| F(s,\cdot)\|_{\mathbb{H}^{1,p}}^q ds.
\end{align*}
Thanks to a Gronwall lemma and the Sobolev embedding $\mathcal{C}^{1,\varepsilon}_b\subset W^{2-2/q+\varepsilon,p}$, we deduce
\begin{equation}\label{eq:supBound}
\sup_{t\in[0,T]} \|F(t,\cdot)\|_{\mathbb{H}^{1,p}} \leq \left( C_{T}\|f\|_{\mathbb{D}^{1,0,p}_q} + C\|\phi\|_{\D^{1,2-2/q+\varepsilon,p}}\right)e^{\frac {C_{T}T} q\|\tilde{b}\|_{L^q([0,T])}^q}.
\end{equation}
We now turn to Estimate \eqref{eq:ParabolicNadaptEsti}. We can apply the $\D^{1,2,p}_q$-norm to \eqref{eq:ParabolicNadapt} and obtain, by using lemma \ref{lem:EstimChaleur},
\begin{align*}
\|F\|_{\D^{1,2,p}_q}^q \leq &C_{T}\|\phi\|_{\D^{1,2-2/q+\varepsilon,p}}^q + C \|f\|_{\D^{1,0,p}_q}^q + C \int^T_t \|b(s,\cdot)\cdot \nabla F(s,\cdot)\|_{\D^{1,0,p}}^qds
\\ \leq & C_{T}\|\phi\|_{\D^{1,2-2/q+\varepsilon,p}}^q + C\|f\|_{\D^{1,0,p}_q}^q + C\int^T_t |\tilde{b}(s)|^q \| F(s,\cdot)\|_{\mathbb{H}^{1,p}}^qds
\end{align*}
which yields, thanks to \eqref{eq:supBound},
\begin{equation}\label{eq:DBoundPrelimin}
\|F\|_{\D^{1,2,p}_q}^q \leq C_{T}(1+\|\tilde{b}\|_{L^q([0,T])}^qe^{C_{T}T\|\tilde{b}\|_{L^q([0,T])}^q})\left(\|\phi\|_{\D^{1,2-2/q+\varepsilon,p}}^q + \|f\|_{\D^{1,0,p}_q}^q\right),
\end{equation}
Then, we differentiate \eqref{eq:ParabolicNadapt} with respect to the time variable and deduce the equation
\begin{equation}\label{eq:ParabolicNadaptDiffP}\left\{\begin{array}{ll}
\partial_t F(t,x) = \L^X_{t} F(t,x) + f(t,x),
\\ F(T,x) = \phi(x).
\end{array}\right.\end{equation}
By applying the $\D^{1,0,p}_q$-norm to \eqref{eq:ParabolicNadaptDiffP} and by using the estimate \eqref{eq:supBound}, we obtain
\begin{align*}
\|\partial_t F\|_{\D^{1,0,p}_q} &\leq \frac 1 2 \|\Delta F\|_{\D^{1,0,p}_q}  + \|f\|_{\D^{1,0,p}_q} + \| b\cdot\nabla F\|_{\D^{1,0,p}_q}
\\ &\leq C_{T}\left(\|\phi\|_{\D^{1,2-2/q+\varepsilon,p}}+\|f\|_{\D^{1,0,p}_q}\right),
\end{align*}
which, together with \eqref{eq:DBoundPrelimin}, gives Estimate  \eqref{eq:ParabolicNadaptEsti}. 
\\\textbf{Step 2:} The last argument of the proof consists in using the so-called \textit{continuity method}. For $\mu\in[0,1]$, we consider the equation
\begin{equation}\label{eq:ParabolicNadaptCont}
F_{\mu}(t,x) =P_{t,T}\phi(x) -\int_t^T P_{t,s}f(s,x)ds - \int_t^TP_{t,s}\left[ \mu b(s,x)\cdot\nabla F_{\mu}(s,x) \right]ds.
\end{equation}
We wish to prove that the set $\nu\subset[0,1]$ of elements $\mu$ for which \eqref{eq:ParabolicNadaptCont} admits a unique solution is $[0,1]$ (with $\mu=1$ corresponding to the equation \eqref{eq:ParabolicNadapt}). In the case where $\mu=0$, the existence and uniqueness of a solution of \eqref{eq:ParabolicNadapt} is straightforward and, thus, $\nu$ is not empty. Fix $\mu_0\in\nu$ and denote $\mathcal{R}^{\mu_0}$ the mapping from $\D^{1,0,p}_q$ to $\mathbb{G}^{1,2,p}_q$ which maps $f$ to the solution $F_{\mu_0}$ of \eqref{eq:ParabolicNadaptCont} for $\phi=0$. Let $\mu\in[0,1]$ to be fix later. The existence and uniqueness of the solution of equation \eqref{eq:ParabolicNadaptCont} relies on a fixed point argument. We consider the mapping $\Gamma_{\mu}$ given by
\begin{equation*}
\Gamma_{\mu}(F) = P_{\cdot,T}\phi +\mathcal{R}^{\mu_0}f + (\mu-\mu_0)\mathcal{R}^{\mu_0}\left(b\cdot \nabla F\right),
\end{equation*}
and aim to prove that it is a contraction mapping from $\mathbb{G}^{1,2,p}_q$ to itself. It follows from the estimates \eqref{eq:ParabolicNadaptEsti} and \eqref{eq:EstimCGCauchy} that, $\forall F_1,F_2\in\mathbb{G}^{1,2,p}_q$,
\begin{align*}
\|\Gamma_{\mu}(F_1) - \Gamma_{\mu}(F_2)\|_{\mathbb{G}^{1,2,p}_q}  \leq & C |\mu-\mu_0| \|b\cdot \nabla(F_1-F_2)\|_{\D^{1,0,p}_q}
\\ \leq & C |\mu-\mu_0| \left(\int_0^T |\tilde{b}(s)|^q \|F_1(s,\cdot)-F_2(s,\cdot)\|_{\mathbb{C}^{1,p}}^qds\right)^{1/q}
\\ \leq & C |\mu-\mu_0| \|\tilde{b}\|_{L^q([0,T])}\|F_1-F_2\|_{\mathbb{G}^{1,2,p}_q}.
\end{align*}
Hence, by choosing $\mu$ such that $|\mu-\mu_0|< \frac 1 {C\|\tilde{b}\|_{L^q([0,T])}}$, we can conclude that there exists a unique solution to \eqref{eq:ParabolicNadaptCont}. Therefore, by repeating the argument a finite number of times, we prove that $\nu = [0,1]$ and that \eqref{eq:ParabolicNadapt} admits a unique solution in $\mathbb{G}^{1,2,p}_q$.
\end{proof}

\begin{cor} 
Let $f\in \D_q^{1,0,p}$ and $\phi\in\D^{1,2-2/q+\varepsilon,p}$, with $\varepsilon>0$. Under Assumption \ref{asm:Cauchy}, there exists a unique solution $F$ in $\D_q^{1,2,p}$ to the equation \eqref{eq:FK}.
\end{cor}

\begin{proof}
The existence of the solution follows directly from Proposition \ref{prop:Malliavindifnonadapted} since one can check that a solution of \eqref{eq:ParabolicNadapt} is a solution to \eqref{eq:FK}. To prove the uniqueness, we consider a solution $F$ of \eqref{eq:FK} with $\phi=0$ and $f=0$. Let $F^n$ be a sequence of smooth functions in $(t,x)$ of $\mathbb{G}^{1,2,p}_q$ such that
\begin{equation*}
\|F-F^n\|_{\D^{1,2,p}_q} + \|\partial_t F - \partial_t F^n\|_{\D^{1,0,p}_q}\underset{n\rightarrow\infty}{\longrightarrow} 0.
\end{equation*}
Therefore, we have that
\begin{equation*}
\partial_t F^n(t,x)-\L^X_{t} F^n(t,x) \underset{n\rightarrow\infty}{\longrightarrow} \partial_t F(t,x)+ \L^X_{t} F(t,x)  = 0,
\end{equation*}
in $\D^{1,0,p}_q$. By denoting $\mathcal{R}$ the linear bounded operator from $\D^{1,0,p}_q$ to $\mathbb{G}^{1,2,p}_q$ which associates $f$ with the solution $F$ of \eqref{eq:ParabolicNadapt} and since $\mathcal{R}f$ solves \eqref{eq:FK}, we have a representation of $F^n$ as
\begin{equation}\label{eq:RepresFn}
F^n= \mathcal{R}\left(\partial_t F^n - \mathcal{L}^XF^n \right).
\end{equation}
It follows from \eqref{eq:RepresFn} and \eqref{eq:ParabolicNadaptEsti} that
\begin{equation*}
\|F^n\|_{\mathbb{G}^{1,2,p}_q} \leq C\|\partial_tF^n-\mathcal{L}^XF^n\|_{\D^{1,0,p}_q} \underset{n\rightarrow\infty}{\longrightarrow} 0,
\end{equation*}
which implies that $\|F\|_{\mathbb{G}^{1,2,p}_q} = 0$.
\end{proof}

From now on, we denote $(P_{s,t}^X)_{0\leq s\leq t\leq T}$ the family of linear operators  associated to the solution of the Fokker-Planck equation determined by $\mathcal{L}^X$, that is, $P_{s,t}^X \phi(x)$ is the solution to the SPDE
\begin{equation}
\label{eq:RepresentationPX}
P_{s,t}^X \phi(x) = \phi(x) - \int_s^t \L^X_{r} P_{r,t}^X \phi(x) dr, \quad 0\leq s\leq t,
\end{equation}
with $\phi$ a $\mathcal{F}_t$-measurable mapping in $\D^{1,2-2/q+\varepsilon,p}$. We end this section by the following Lemma which gives some estimates on $P^X$.

\begin{lemma}\label{lem:SmoothPX}
Let $f\in \D_q^{1,0,p}$ and $\phi\in\D^{1,2-2/q+\varepsilon,p}$, with $\varepsilon>0$.
 Under Assumption \ref{asm:Cauchy}, the following estimates hold
\begin{align}
\|P^X_{\cdot,T}\phi\|_{\mathbb{G}^{1,2,p}_q} \leq C_{1,T}\|\phi\|_{\D^{1,2-2/q+\varepsilon,p}},\label{eq:estintPphi}
\\ \left\|\int_{\cdot}^T P^X_{\cdot,r}f(r,\cdot)dr\right\|_{\mathbb{G}^{1,2,p}_q} \leq C_{2,T}\|f\|_{\D^{1,0,p}_q},\label{eq:estintPf}
\end{align}
and
\begin{equation}
\int_0^T \|\mathcal{L}^X_\cdot P^X_{\cdot,r}f(r,\cdot)\|_{\D^{1,0,p}_q}^q dr \leq C \|f\|_{\D^{1,2-2/q+\varepsilon,p}_q}^q.\label{eq:estimFubini}
\end{equation}
\end{lemma}
\begin{proof}
The estimates \eqref{eq:estintPphi} and \eqref{eq:estintPf} are direct consequences of Proposition \ref{prop:Malliavindifnonadapted}. For the last estimate, thanks to \eqref{eq:EstimCGCauchy}, \eqref{eq:MalliavinCalculTilde}, and \eqref{eq:estintPphi}, there exists a constant $C_r>0$ uniformly bounded in $r \in [0,T]$ such that
\begin{equation*}
\|b\cdot \nabla P^X_{\cdot,r}f(r,\cdot)\|_{\D^{1,0,p}_q} \leq C_r \|f(r,\cdot)\|_{\D^{1,2-2/q+\varepsilon,p}}.
\end{equation*}
Therefore, the estimate follows from \eqref{eq:estintPphi} since
\begin{equation*}
\int_0^T \|\mathcal{L}^X_\cdot P^X_{\cdot,r}f(r,\cdot)\|_{\D^{1,0,p}_q}^q dr \leq \int_0^T C_r^q \|f(r,\cdot)\|_{\D^{1,2-2/q+\varepsilon,p}}^qdr.
\end{equation*}
\end{proof}

We can also compute the Malliavin derivative of $(P_{s,t}^X)_{0\leq s\leq t\leq T}$. This is goal of the next lemma.
\begin{lemma}\label{lem:MalliavinBCalcu}
We have the following commutation formula between the Malliavin derivative and the operator $P^X$
\begin{equation}
D_t P^X_{t,T}\phi(x) = P^X_{t,T} D_t \phi(x) - \int_t^T P^X_{t,r}\left( D_t b(r,x)\cdot \nabla P^X_{r,T}\phi(x)\right)dr.
\end{equation}
\end{lemma}

\begin{proof}
Let $t\leq r\leq T$. Denote
\begin{equation*}
\Phi(r,x) : = D_t P^X_{r,T}\phi(x),
\end{equation*}
then, a direct computation of the Malliavin derivative applied to the representation formula of $P^X$ gives
\begin{equation*}
\Phi(r,x) = \Phi(T,x) - \int_r^T \L^X_{u}\Phi(u,x) du - \int_r^T D_t b(u,x)\cdot \nabla P^X_{u,T}\phi(x) du.
\end{equation*}
Hence, by the representation formula of $P^X$, we deduce the following \textit{mild} formulation of $\Phi$
\begin{equation*}
\Phi(r,x) = P^X_{r,T} \Phi(T,x) - \int_r^T P^X_{r,u}\left(D_t b(u,x)\cdot \nabla P^X_{u,T}\phi(x)\right) du,
\end{equation*}
and, thus, the desired result.
\end{proof}
\subsection{The adapted Fokker-Planck equation}
\label{section:adaptedFK}

We consider now the following BSPDE:
\begin{equation}
\label{eq:BSPDE}
F(t,x) = - \int_t^T \left(\L^X_{r} F(r,x) + f(r,x) \right)dr - \int_t^T Z(r,x) dW_r, 
\end{equation}
where $f$ belongs to $\D^{1,0,p}_q$. To ensure the existence of such representation, we need to work under the natural filtration $(\mathcal{F}_t^W)_{t\in [0,T]}$ of the Brownian motion $W$. In this section we will call a predictable process a $(\mathcal{F}_t^W)_{t\in [0,T]}$-predictable stochastic process. Note that by the definition of a weak solution (\textit{cf.} Definition \ref{definition:weaksol}) as $\mathcal{F}^W_\cdot \subset \mathcal{F}_\cdot$, any $(\mathcal{F}^W_t)_{t\in [0,T]}$-predictable process is $(\mathcal{F}_t)_{t\in [0,T]}$-predictable.

From now on we assume that:

\begin{asm}
\label{asm:fadpat} $f$ is a predictable field.
\end{asm}

Before going further we recall what is a solution to the BSPDE \eqref{eq:BSPDE} in our context. To this end we say that a random field $\varphi:\Omega\times[0,T]\times\real^d \to \real$ is predictable if for any $x$ in $\real^d$, $\varphi(\cdot,x)$ is predictable. We set for $m\in \mathbb{N}$:

$$ \mathbb{W}^{m,p}_{\mathcal{P},q}:=\{\varphi \textrm{ predictable field }, \; \|\varphi\|_{\W_q^{m,p}}<+\infty \}, $$
$$ \mathbb{M}^{p}:=\left\{\varphi \textrm{ predictable field }, \; \int_{\real^d} \E\left[\left(\int_0^T |\varphi(s,x)|^2 dt\right)^\frac{p}{2}\right] dx<+\infty \right\}. $$

\begin{defi}[Adapted strong solution to a BSPDE]
\label{def:BSPDE}
We say that a pair of predictable random fields $(F,Z)$ is strong solution to the BSPDE \eqref{eq:BSPDE} if 
\begin{equation*}
(F,Z) \in \mathbb{W}^{2,p}_{\mathcal{P},q} \times  \mathbb{M}^{p}
\end{equation*}
with $\frac d p + \frac 2 q <1$ and Relation \eqref{eq:BSPDE} is satisfied for every $t$ in $[0,T]$, for a.e. $x$ in $\real^d$, $\P$-a.s.. 
\end{defi}

\begin{remark}
\label{rem:BSPDE}
We warn the reader that in the previous definition, the \textit{predictable} feature of the fields $(F,Z)$ is crucial. In that sense we will speak of BSPDE. This differs from the SPDE \eqref{eq:FK} whose solution is not adapted (see Remark \ref{rem:non-predictable}). In that case we will speak of SPDEs to emphasis that the measurability requirement is not present.
\end{remark}

In order to proceed further, we need some additional assumptions on the Malliavin derivatives of $f$ and $b$.
\begin{asm}\label{asm:MalliavinFB}
Let $m\in [q,\infty]$ and $\ell\in[p,\infty]$ such that
\begin{equation*}
\frac 1 m + \frac 1 {\bar{m}} = \frac 1 q \quad \mbox{ and } \quad \frac 1 \ell + \frac 1 {\bar{\ell}} = \frac 1 p.
\end{equation*}
We assume that there exist a function $f'\in L^{m}([0,T];L^{\ell}(\Omega; L^p(\mathbb{R}^d)))$ (resp. $b'$) and a function $m_f\in L^{\bar{m}}([0,T];L^{\bar{\ell}}([0,T]\times\Omega))$ (resp. $m_b$) such that
\begin{equation*}
D_{\theta} f(t,x) = f'(t,x) m_f(\theta,t)
\end{equation*}
Moreover, we assume that, for a.e. $t\in[0,T]$, $\partial_{t} m_f(t,\cdot)$ (resp. $\partial_{t} m_b(t,\cdot)$) is a measure on $[0,T]$ and that there exists a constant $C>0$ such that
\begin{equation*}
\left\|\int^T_{\cdot} P^X_{\cdot,s} f'(s,\cdot)\partial_{t}m_f(\cdot,ds)\right\|_{\mathbb{W}^{0,p}_q} \leq C \|f'\|_{\mathbb{W}^{0,p}_q}.
\end{equation*}
Finally, we assume that $\mbox{Tr}(m_f)(t) : = m_f(t,t)$ (resp. $\mbox{Tr}(m_b)$) belongs to $L^{\bar{m}}([0,T];L^{\bar{\ell}}(\Omega))$.
\end{asm}
\begin{remark}
We can see that, under the previous assumption, we have, thanks to Hölder inequality's,
\begin{equation*}
\|f\|_{\D^{1,0,p}_q} \leq \|f\|_{\mathbb{W}^{0,p}_q} + \|f'\|_{L^{m}([0,T];L^{\ell}(\Omega; L^p(\mathbb{R}^d)))}\|m_f\|_{L^{\bar{m}}([0,T];L^{\bar{\ell}}([0,T]\times\Omega))}.
\end{equation*}
Obviously, the same holds for $b$.
\end{remark}

We have the following result concerning the existence and uniqueness of a strong solution to Equation (\ref{eq:BSPDE}).

\begin{theorem}\label{prop:ExistAdapt}
Assume that $f$ belongs to $\D^{1,0,p}_q$ and that Assumption \ref{asm:fadpat} is in force. 
There exists a unique strong (predictable) solution to Equation \eqref{eq:BSPDE}
$$(F,Z)\in \left(\mathbb{W}^{2,p}_{\mathcal{P},q}\right)^2.$$
Futhermore, we have the following representation of $F$
\begin{equation}
\label{eq:repF}
F(t,x) =  \E\left[-\int_t^T P^X_{t,r}f(r,x)dr \Big\vert \mathcal{F}_t\right].
\end{equation}
In addition, for a.e. $(t,x)$, $F(t,x)$ is Malliavin differentiable ($\|F\|_{\D_q^{1,2,p}}<+\infty$), and for a.e. $x\in\mathbb{R}^d$, a version of the process $(Z(t,x))_{t\in[0,T]}$ is given by
\begin{equation}
\label{eq:ZMalliavindiff}
Z(t,x) =  \E\left[-\int_t^T D_t P^X_{t,r}f(r,x)dr \Big\vert \mathcal{F}_t\right].
\end{equation}
Finally, $F$ admits the following mild representation
\begin{equation}\label{eq:MildF}
F(t,x) = -\int_t^T P^X_{t,r} f(r,x)dr - \int_t^T P^X_{t,r}Z(r,x) dW_r.
\end{equation}
\end{theorem}
\begin{proof}
Throughout Step 1 and Step 2, we assume that $f$ and $f'$ are smooth with respect to $x$. Since the norms of $F$ and $Z$ in $\mathbb{W}^{2,p}_q$ are bounded by the norms of $f\in \mathbb{W}^{0,p}_q$ and $f'\in L^{m}([0,T];L^{\ell}(\Omega; L^p(\mathbb{R}^d)))$ (see Step 1 and Step 2), we can consider two sequences of smooth approximations $(f_{n})_{n\in\mathbb{N}}$ and $(f'_n)_{n\in\mathbb{N}}$ such that the limit $(F_n,Z_n)\underset{n\rightarrow\infty}{\longrightarrow}(F,Z)$ converges in $\mathbb{W}^{2,p}_q$. Moreover, thanks to the mild formulation \eqref{eq:MildF}, we obtain that $(F,Z)$ is the unique solution of the Equation \eqref{eq:BSPDE}.

\noindent
\textbf{Step 1:} 
Set 
\begin{equation}
F(t,x):=  \E\left[-\int_t^T P^X_{t,r}f(r,x)dr \Big\vert \mathcal{F}_t\right],
\end{equation}
We start with proving that $F$ belongs to $\mathbb{W}^{2,p}_{\mathcal{P},q}$. Indeed, by using \eqref{eq:estintPf} and Jensen's inequality, it holds that
\begin{align}
\|F(t,\cdot)\|_{\D^{1,2,p}}^p =& \left\|\mathbb{E}\left[ -\int_t^T P^X_{t,s}f(s,\cdot) ds\Big\vert \mathcal{F}_t\right]\right\|_{\D^{1,2,p}}^p\nonumber
\\ = &\hspace{0.2em} \mathbb{E}\left[\left\|\mathbb{E}\left[ -\int_t^T P^X_{t,s}f(s,\cdot) ds\Big\vert \mathcal{F}_t\right]\right\|^p_{W^{2,p}} \right]\nonumber
\\ &+ \int_0^T \mathbb{E}\left[\left\|D_{\theta}\mathbb{E}\left[ -\int_t^T P^X_{t,s}f(s,\cdot) ds\Big\vert \mathcal{F}_t\right]\right\|^p_{W^{2,p}} \right] d\theta\nonumber
\\ \leq&\hspace{0.2em}\mathbb{E}\left[\left\|\int_t^T P^X_{t,s}f(s,\cdot) ds\right\|^p_{W^{2,p}} \right]\nonumber
\\ &+  \int_0^t \mathbb{E}\left[\left\|\int_t^T D_{\theta} P^X_{t,s}f(s,\cdot) ds\right\|^p_{W^{2,p}}\right] d\theta \nonumber
\\ \leq& \left\|\int_t^T P^X_{t,s}f(s,\cdot) ds\right\|^p_{\D^{1,2,p}}<+\infty.\label{eq:BoundFW2p}
\end{align}
We now turn to the derivation of $Z$. We have
\begin{equation*}
- \int_t^T \left(\mathcal{L}_{s}^X F(s,x) +f(x,s)\right) ds = \int_t^T \E\left[\int_s^T \mathcal{L}^X_{s} P^X_{s,r}f(r,x)dr-f(s,x)\Big\vert \mathcal{F}_s\right]ds.
\end{equation*}
By denoting
\begin{equation*}
m(s,x) : =  \int_s^T \mathcal{L}^X_{s} P^X_{s,r}f(r,x)dr -f(s,x),
\end{equation*}
we have that, thanks to the representation \eqref{eq:RepresentationPX},
\begin{align*}
\int_t^T \E\left[m(s,x)\Big\vert \mathcal{F}_t\right] ds &= \E\left[\int_t^T \int_s^T \L^X_{s} P^X_{s,r}f(r,x)drds - \int_t^T f(s,x) ds\Big\vert \mathcal{F}_t\right] 
\\ &= \E\left[\int_t^T \int_t^r \L^X_{s} P^X_{s,r}f(r,x)dsdr - \int_t^T f(s,x) ds\Big\vert \mathcal{F}_t\right] 
\\ &=\E\left[\int_t^T\left(-P^X_{t,r}f(r,x) + f(r,x)\right)dr- \int_t^T f(s,x) ds\Big\vert \mathcal{F}_t\right]
\\ &= F(t,x).
\end{align*}
In the previous computations, we have used Fubini's theorem, which can be applied since, thanks to Lemma \ref{lem:SmoothPX},
\begin{align}
\label{eq:estmart1}
\int_t^T \int_t^r \|\L^X_{s} P^X_{s,r}f(r,\cdot)\|_{\D^{1,0,p}} dsdr &\leq \left(\int_0^T \int_0^T \|\L^X_{s} P^X_{s,r}f(r,\cdot)\|_{\D^{1,0,p}}^q dsdr \right)^{1/q}\nonumber
\\ &\leq  C \|f\|_{\D^{1,2-2/q+\varepsilon,p}_q}.
\end{align}
This enables us to conveniently express the martingale that we are looking for being able to define the field $Z$. That is, we have
\begin{equation*}
F(t,x) = - \int_t^T \left(\mathcal{L}_{s}^X F(s,x) +f(x,s)\right) ds - M(T,x) + M(t,x),
\end{equation*}
where
\begin{equation*}
M(t,x) : = \int_0^t \E\left[m(s,x)\Big\vert \mathcal{F}_s\right] ds +\int_t^T \E\left[m(s,x)\Big\vert \mathcal{F}_t\right]ds.
\end{equation*}
Let us now check that $M$ is indeed a $L^p(\real^d)$-valued martingale. Note first that by Estimate \eqref{eq:estmart1} $M(T,\cdot)$ is integrable as
\begin{align*}
\E\left[\|M(T,\cdot)\|_{L^p(\real^d)}^p\right]&=\E\left[\left\|\int_0^T \E\left[m(s,\cdot)\Big\vert \mathcal{F}_s\right] ds\right\|_{L^p(\real^d)}^p\right]\\
&\leq C \int_0^T \E\left[\|m(s,\cdot)\|_{L^p(\real^d)}^p\right] ds < +\infty,
\end{align*}
since $m$ belongs to $\D_q^{1,0,p}$ (by \eqref{eq:estmart1} and by our assumption on $f$). In addition, $\forall u\in[0,t]$, we have 
\begin{align*}
\E\left[M(t,\cdot)-M(u,\cdot)\Big\vert \mathcal{F}_u\right] =& \int_u^t\E\left[ m(s,\cdot)\Big\vert \mathcal{F}_u\right] +\int_t^T\E\left[ m(s,\cdot)\Big\vert \mathcal{F}_u\right] ds-\int_u^T\E\left[ m(s,\cdot)\Big\vert \mathcal{F}_u\right] 
\\ =& \hspace{0.1cm}0,
\end{align*}
therefore, $M$ is indeed a martingale. It remains to represent $M$ as a stochastic integral against the Brownian motion $W$. For any fixed $x$ in $\real^d$, martingale representation theorem (for real-valued martingales) gives that there exists $Z(\cdot,x):=(Z(t,x))_{t\in[0,T]}$ such that 
$$ \E\left[\int_0^T |Z(t,x)|^2 dt\right]<+\infty,$$
and $M(t,x) =M(0,x) +\int_0^t Z(s,x) dW_s, \quad \forall t\in [0,T], \;\P-a.s..$
Note however, that the subset of $\Omega$ where the equality may fail depends \textit{a priori} on $x$. To obtain, a representation for $L^p(\real^d)$-valued martingales (that is for every $t$, and a.e. x, $\P$-a.s.) we need some extra regularity on $Z$ that we provide here. Set:
$$ \tilde{M}(t,x):=M(0,x) +\int_0^t Z(s,x) dW_s, \quad \forall (t,x). $$
We claim that $\tilde{M}$ is a $L^p(\real^d)$-valued martingale. Indeed using the Burkholder-Davis-Gundy inequality for real-valued martingales, 
\begin{align*}
\E\left[\|\tilde{M}(T,\cdot)-M(0,x)\|_{L^p(\real^d)}^p\right]&=\int_{\real^d} \E\left[\left|\int_0^T Z(t,x) dW_t\right|^p\right] dx\\
& \leq C_{ BDG} \int_{\real^d} \E\left[\left(\int_0^T |Z(t,x)|^2 dt\right)^{p/2}\right] dx\\
& \leq C \int_{\real^d} \E\left[|M(T,x)-M(0,x)|^{p}\right] dx\\
& = C \E\left[\|M(T,x)\|_{L^p(\real^d)}^{p}\right] <+\infty.
\end{align*} 
In particular, $Z$ belongs to $\mathbb{M}^{p}$. Note that once this integrability property is proved for $\tilde{M}$, its martingale feature is straightforward.  
Using Doob's inequality for $L^p(\real^d)$-valued martingales, we get that:
\begin{align*}
\E\left[\sup_{t\in [0,T]} \|M(t,\cdot)-\tilde{M}(t,\cdot)\|_{L^p(\real^d)}^p\right] &\leq C \sup_{t\in [0,T]} \E\left[\|M(t,\cdot)-\tilde{M}(t,\cdot)\|_{L^p(\real^d)}^p\right] \\
&\leq C \E\left[\|M(T,\cdot)-\tilde{M}(T,\cdot)\|_{L^p(\real^d)}^p\right]=0,
\end{align*}
by definition of $\tilde{M}$. This proves that 
$$ M(t,x) = M(0,x) + \int_0^t Z(s,x) dW_s, \quad \forall t, \; \textrm{ for a.e. } x, \; \P-a.s.. $$
Thus, we obtain that $(F,Z) \in \mathbb{W}^{2,p}_{\mathcal{P},q}\times\mathbb{M}^{p}$ solves Equation \eqref{eq:BSPDE}.\vspace{0.5em}

\noindent
\textbf{Step 2:} Proof of \eqref{eq:ZMalliavindiff}.
\vspace{0.5em}

\noindent
Recall that by \eqref{eq:BoundFW2p}, $\|F(t,\cdot)\|_{\D^{1,2,p}}<+\infty$. In addition, following the same lines as in the computation of \eqref{eq:BoundFW2p}, we have that:
\begin{align*}
\left\| \int_t^T \L^X_{r}F(r,\cdot) dr \right\|_{\D^{1,0,p}_q}^q &= \int_0^T \left\| \int_t^T \L^X_{r} F(r,\cdot) dr \right\|_{\D^{1,0,p}}^q dt\\
&\leq T \int_0^T \left\| \L^X_{r} \int_r^T P^X_{r,s}f(s,\cdot)ds \right\|_{\D^{1,0,p}}^q dr\\
&\leq T \int_0^T \int_0^s \left\| \L^X_r P^X_{r,s} f(r,\cdot) \right\|_{\D^{1,2,p}}^q dr ds, \textrm{ by Lemma \ref{lem:SmoothPX}} \\
&\leq C T \int_0^T \left\|f(r,\cdot) \right\|_{\D^{1,2-2/q+\varepsilon,p}}^q dr = CT \|f\|_{\D^{1,2-2/q+\varepsilon,p}_q}^q.
\end{align*}
Combining this result with Relation \eqref{eq:BSPDE}, we obtain that for a.e. $(t,x)$, $\int_t^T Z(s,x) dW_s$ belongs to $\D^{1,p}$ (see Remark \ref{rem:MalliavinIncl}). Since $\D^{1,p}\subset \D^{1,2}$, by \cite[Lemma 2.3]{PardouxPeng_92}, this is equivalent to for a.e. $(t,x)$, $Z(\cdot,x) \in L^2([t,T],\D^{1,2})$. As a consequence, for a.e. $(t,x)$ and for any $0\leq s \leq t$,
$$ D_s F(t,x) = \int_s^t D_s \L^X_{r} F(r,x) + D_s f(r,x)dr + Z(s,x) + \int_s^t Z(r,x) dW_r, \; \P-a.s..$$
Hence taking $s=t$, in the previous relation, we have that for a.e. $x$, a version of the process $(Z(t,x))_{t\in [0,T]}$ is given by $Z(t,x) = D_t F(t,x)$. Representation \eqref{eq:ZMalliavindiff} can then be deduced using \cite[Proposition 1.2.8]{Nualartbook}. We are now in position to prove that $Z$ belongs to $\mathbb{W}^{2,p}_q$. By using Lemma \ref{lem:MalliavinBCalcu} and Assumption \ref{asm:MalliavinFB}, we have
\begin{align}
D_t F(t,x) = \E\left[-\int_t^T D_t \right. &\hspace{-0.4em}\left.P^X_{t,r}f(r,x)dr \Big\vert \mathcal{F}_t\right] = \E\left[-\int_t^T  m_f(t,r) P^X_{t,r} f'(r,x)dr \Big\vert \mathcal{F}_t\right]\nonumber
\\ &+\E\left[\int_t^T \int_t^r  m_b(t,u)P^X_{t,u}\left(b'(u,x)\cdot \nabla P^X_{u,r} f(r,x)\right)dudr \Big\vert \mathcal{F}_t\right].\label{eq:RepresentZW}
\end{align}
By differentiating with respect to the time variable, it follows that
\begin{align*}
-\int_t^T m_f(t,r) P^X_{t,r}f'(r,x) dr =& - \int_t^T  P^X_{t,r} (m_f(r,r) f'(r,x) )dr 
\\ &+ \int_t^T P^X_{t,r}\left(\int_r^T P^X_{r,s} f'(s,x) \partial_{r}m_f(r,ds) \right)dr.
\end{align*}
Hence, by Assumption \ref{asm:MalliavinFB} and Lemma \ref{lem:SmoothPX}, we estimate the first term on the rhs of \eqref{eq:RepresentZW}
\begin{align*}
\left\| \E\left[-\int_{\cdot}^T  m_f(\cdot,r) P^X_{\cdot,r} f'(r,\cdot)dr \Big\vert \mathcal{F}_{\cdot}\right] \right\|_{\mathbb{W}^{2,p}_q} \leq &  \left\|\int_{\cdot}^T  m_f(\cdot,r) P^X_{\cdot,r} f'(r,\cdot)dr\right\|_{\mathbb{W}^{2,p}_q}
\\ \leq &C \hspace{0.2em} \|\mbox{Tr}(m_f) f'\|_{\mathbb{W}^{0,p}_q} 
\\ &+ C\left\|\int_{\cdot}^{T} P^X_{\cdot,s} f'(s,\cdot)\partial_{r}m_f(\cdot,ds)\right\|_{\mathbb{W}^{0,p}_q}
\\  \leq &\hspace{0.2em} \|f'\|_{L^{m}([0,T];L^{\ell}(\Omega; L^p(\mathbb{R}^d)))}
\\ &\times C(\|\mbox{Tr}(m_f) \|_{L^{\bar{m}}([0,T];L^{\bar{\ell}}(\Omega))} + 1)\left\|f'\right\|_{\mathbb{W}^{0,p}_q}.
\end{align*}
For the second term of \eqref{eq:RepresentZW}, we remark that, thanks to Fubini's theorem,
\begin{equation*}
\int_t^T \int_t^r  m_b(t,u)P^X_{t,u}\left(b'(u,x)\cdot \nabla P^X_{u,r} f(r,x)\right)dudr =  \int_t^T P^X_{t,u} m_b(t,u) G(u,x)du,
\end{equation*}
where we denoted $G(u,x) : = b'(u,x)\cdot\nabla \left[\int_u^T P^X_{u,r} f(r,x) dr\right]$. Hence, we can proceed by similar arguments as for the first term of the rhs of \eqref{eq:RepresentZW} since, by \eqref{eq:MalliavinCalculTilde}, \eqref{eq:EstimCGCauchy} and Lemma \ref{lem:SmoothPX},
\begin{equation*}
\|G\|_{\mathbb{W}^{0,p}_q}\leq C \left\|\int_u^T P^X_{u,r} f(r,x) dr\right\|_{\mathbb{G}^{0,2,p}_q} \leq C \|f\|_{\mathbb{W}^{0,p}_q}<\infty.
\end{equation*}
Therefore, we conclude that $D_tF(t,x)$ belongs to $\mathbb{W}^{2,p}_q$ and, thus, $Z$ itself belongs to $\mathbb{W}^{2,p}_q$.
\vspace{0.5em}

\noindent
\textbf{Step 3:} Uniqueness of the solution.\\\vspace{0.5em}

\noindent
Assume that there exist two solutions in $\mathbb{W}^{2,p}_{\mathcal{P},q}$ to the BSPDE \eqref{eq:BSPDE}. Then by linearity, the difference of these solutions is itself solution to the BSPDE with $f \equiv 0$. Let $(\hat{F},\hat{Z})$ be any solution of \eqref{eq:BSPDE} with $f \equiv 0$. We will prove that $\hat{F}\equiv \hat{Z}\equiv 0$ in $(\mathbb{W}^{2,p}_{\mathcal{P},q})^2$ which will prove the claim. To this end, let $\theta:\real^d \to \real$ be a non-negative smooth bump function such that $\theta(x)=1$ if $|x|\leq 1$ and $\theta(x)=0$ if $|x|\geq 2$. For any positive integer $n$ we set $\theta^n(x):=\theta(x/n)$, $F^n(t,x):=\hat{F}(t,x) \theta^n(x)$, and $Z^n(t,x):=\hat{Z}(t,x) \theta^n(x)$. By definition, we have that
\begin{align*}
&\E\left[\int_0^T \|F^n(t,\cdot)\|_{W^{2,p}}^2 dt + \int_0^T \|Z^n(t,\cdot)\|_{W^{2,p}}^2 dt\right]\\
&\leq C \int_0^T \E\left[ \|F^n(t,\cdot)\|_{W^{2,p}}^p \right]^{2/p} dt + \int_0^T \E\left[\|Z^n(t,\cdot)\|_{W^{2,p}}^p \right]^{2/p}dt \\
&\leq C (\|F^n\|_{\mathbb{W}^{2,p}_{q}}^2+\|Z^n\|_{\mathbb{W}^{2,p}_{q}}^2)<+\infty,
\end{align*}
by Jensen's inequality.
In addition,
$$ \sup_x |\nabla \theta^n(x)| \leq n^{-1} \|\theta\|_\infty, \quad \sup_x |\Delta \theta^n(x)| \leq n^{-2} \|\theta\|_\infty.$$
In addition, since $\theta^n$ is a smooth function it follows that $(F^n,Z^n)$ is solution to the BSPDE:
$$ F^n(t,x) = 0 - \int_t^T \L^X_{r} (F^n(r,x)) + \psi^n(r,x) dr - \int_t^T Z^n(r,x) dW_r, $$
where $\psi^n(r,x):=-(\nabla F \cdot \nabla \theta^n+\frac12 F \Delta \theta^n + F b \cdot \nabla \theta^n)(r,x)$. Recall that $p,q\geq 2$ so that we can make use of a priori estimates in $L^2$ as \cite[Theorem 2.2]{du2010revisit} to obtain that there exists a universal constant $C>0$ such that:
\begin{equation}
\label{eq:Kaiest}
\E\left[\int_0^T \|F^n(t,\cdot)\|_{W^{2,p}}^2+ \|Z^n(r,\cdot)\|_{W^{1,p}}^2 dt\right]\leq C \E\left[\int_0^T \|\psi^n(t,\cdot)\|_{L^2_x} dt\right]. 
\end{equation}
We estimate the right-hand side of the previous estimate. 
For the first term, we have:
\begin{align*}
&\E\left[\int_0^T \int_{\real^d} |\nabla F(t,x) \cdot \nabla \theta^n(x)|^2 dx dt\right]\\
&\leq \|\theta\|_\infty n^{-2} \E\left[\int_0^T \int_{\mathcal{B}(0,n)} |\nabla F(t,x)|^2 dx dt\right]\\
&\leq C n^{-(p+2)/p} \E\left[\int_0^T \left(\int_{\real^d} |\nabla F(t,x)|^p dx\right)^{2/p} dt\right]\\
&\leq C n^{-(p+2)/p} \int_0^T \E\left[ \|F(t,x)\|_{W^{2,p}}^p \right]^{2/p} dt\\
&\leq C n^{-(p+2)/p} \left(\int_0^T \E\left[ \|F(t,x)\|_{W^{2,p}}^p \right]^{q} dt\right)^{2/(qp)}\\
&= C n^{-(p+2)/p},
\end{align*}
where we have used H\"older inequality several times, the fact that $q,p\geq 2$. Similar calculations for the two other terms involved in the definition of $\psi^n$ lead to:
$$ \lim_{n\to +\infty} \E\left[\int_0^T \|\psi^n(t,\cdot)\|_{L^2_x} dt\right]=0, $$
which implies that:
$$ \lim_{n\to+\infty} \E\left[\int_0^T \|F^n(t,\cdot)\|_{W^{2,p}}^2+ \|Z^n(t,\cdot)\|_{W^{1,p}}^2 dt\right] = 0,$$
in view of the Estimate \eqref{eq:Kaiest}. As a consequence, we can deduce that:
$$ \lim_{n\to+\infty} \E\left[\int_0^T \|\hat{F}(t,\cdot) \textbf{1}_{\mathcal{B}(0,n)} \|_{W^{2,p}}^2 + \|\hat{Z}(t,\cdot) \textbf{1}_{\mathcal{B}(0,n)}\|_{W^{1,p}}^2 dt\right] = 0,$$
which implies that:
$$ \E\left[\int_0^T \|\hat{F}(t,\cdot)\|_{W^{2,p}}^2 + \|\hat{Z}(t,\cdot)\|_{W^{1,p}}^2 dt\right] = 0.$$
Hence $(\hat{F},\hat{Z})\equiv(0,0)$ in $\mathbb{W}^{2,p}_{\mathcal{P},q}$.\vspace{0.5em}

\noindent
\textbf{Step 4:} Proof of the mild representation \eqref{eq:MildF}.\\\vspace{0.5em}
Set:
$$ \tilde{F}(t,x) = - \int_t^T P^X_{t,r} f(r,x) dr - \int_t^T P^X_{t,r} Z(r,x) dW_r, \quad t\in [0,T],$$
where $Z$ is the second component of the solution to Equation (\ref{eq:BSPDE}). We wish to prove that $\tilde{F}\in \mathbb{W}^{2,p}_{q}$ is the first component of the solution to Equation (\ref{eq:BSPDE}) (\textit{i.e.} $\tilde{F} = F$). Here we stress that we do not impose $\tilde{F}$ to be predictable. We have, by Burkholder-Davis-Gundy's inequality for real-valued martingales and Lemma \ref{lem:SmoothPX},
\begin{align*}
\left\|\int_{\cdot}^T P^{X}_{\cdot,r}Z(r,\cdot) dW_r\right\|_{\W^{2,p}_q}^q &\leq C \int_0^T \left(\left\|\int_t^T \left| P^X_{t,r}Z(r,\cdot)\right|^2dr \right\|_{\mathbb{W}^{2,p/2}}^{1/2}\right)^q dt\\
& \leq C \int_0^T \left(\int_t^T \left\|P^X_{t,r}Z(r,\cdot) \right\|_{\mathbb{W}^{2,p}}^{2} dr \right)^{q/2} dt\\
& \leq C \int_0^T \int_0^r \left\|P^X_{t,r}Z(r,\cdot) \right\|_{\mathbb{W}^{2,p}}^{q} dt dr\\
& \leq C \|Z\|_{\mathbb{W}^{2,p}_{q}}^q<+\infty,
\end{align*}
which yields that the stochastic integral is well-defined and $\tilde{F}$ belongs to $\mathbb{W}^{2,p}_{q}$.
With the definition of $P^X$ (see \eqref{eq:RepresentationPX}), we decompose $\tilde{F}$ as follows:
\begin{align*}
\tilde{F}(t,x)&= \int_t^T \int_t^r \L^X_{u}P^X_{u,r} f(r,x) du dr - \int_t^T f(r,x) dr \\
&\quad + \int_t^T \int_t^r \L^X_{u}P^X_{u,r} Z(r,x) du dW_r - \int_t^T Z(r,x) dW_r 
\end{align*}
Using Stochastic Fubini's Theorem (that we justify below), we have that:
\begin{align*}
\tilde{F}(t,x) &= \int_t^T \L^X_{u}\int_u^T P^X_{u,r} f(\cdot,r) dr du - \int_t^T f(r,x) dr  \\
&\quad +\int_t^T \L^X_{u}\int_u^T P^X_{u,r} Z(r,x) dW_r du - \int_t^T Z(r,x) dW_r \\
&=- \int_t^T \L^X_{u}\underbrace{\left[ - \int_u^T P^X_{u,r} f(x,r) dr - \int_u^T P^X_{u,r} Z(r,x) dW_r\right]}_{\tilde{F}(u,x)} du \\
&\quad  - \int_t^T f(r,x) dr - \int_t^T Z(r,x) dW_r.
\end{align*}
This computation proves that $\tilde{F}$ is solution to the (non-adapted) SPDE:
$$ \tilde{F}(t,x) = -\int_t^T \L^X_{u}\tilde{F}(u,x) + f(u,x) du - \int_t^T Z(r,x) dW_r,$$
where $- \int_t^T Z(r,x) dW_r$ is seen as a source term. By definition, $F$ is also a solution to this equation. As a consequence $\hat{F}(t,x):=F(t,x)-\tilde{F}(t,x)$  is solution (in $\mathbb{W}^{2,p}_{q}$) to the SPDE
$$ \hat{F}(t,x) = -\int_t^T \L^X_{u}\hat{F}(u,x) du, $$
which admits $0$ as unique solution in $\mathbb{W}^{2,p}_{q}$ (by Proposition \ref{prop:Malliavindifnonadapted}), which proves that $F=\tilde{F}$ in $\mathbb{W}^{2,p}_{q}$. We finally justify the use of stochastic's Fubini theorem. More precisely, we have that:
\begin{align*}
&\left\|\int_0^T \int_r^T |\L^X_{u} P^X_{u,r} Z(r,\cdot)|^2 du dr\right\|_{\mathbb{W}^{0,p/2}}\\
&\leq \int_0^T \int_r^T \left\| \L^X_{u} P^X_{u,r} Z(r,\cdot) \right\|_{\mathbb{W}^{0,p}}^2 du dr\\
&\leq C \left(\int_0^T \int_r^T \left\| \L^X_{u} P^X_{u,r} Z(r,\cdot) \right\|_{\mathbb{W}^{0,p}}^{q} du dr\right)^{2/q}\\
&\leq C \|Z\|_{\mathbb{W}_q^{2,p}}^{2}.
\end{align*}
\end{proof}

\section{The It\^o-Tanaka-Wentzell trick and some applications}
\label{section:ITK}

\subsection{Main result}

Let us recall the It\^o-Wentzell formula in the context of processes with values in Sobolev spaces \cite{ItoKrylov}. 

\begin{prop}[It\^o-Wentzell formula]\label{prop:ItoWentzell}
Let $F$ in $\mathbb{W}^{2,p}_{\mathcal{P},q}$ be such that for any $\varphi\in L^{\bar{p}}(\mathbb{R}^d)$: 
\begin{equation}
\label{eq:decF}
\langle F(t,\cdot), \varphi \rangle =\langle F(0,\cdot), \varphi \rangle +\int_0^t \langle \Gamma(s,\cdot), \varphi \rangle dW_s + \int_0^t \langle A(s,\cdot), \varphi \rangle ds
\end{equation}
with $F(0,\cdot)\in L^p(\mathbb{R}^d)$, $A$ in $\mathbb{W}^{0,p}_{\mathcal{P},q}$ and $\Gamma$ in $\mathbb{W}^{1,p}_{\mathcal{P},q}$. 
Then, $\forall t\in[0,T]$, $\forall \varphi\in L^{\bar{p}}(\mathbb{R}^d)$,
\begin{align}
\langle F(t,\cdot+X_t) , \varphi\rangle= & \langle F(0,\cdot +X_0) , \varphi\rangle  +\int_0^t [ \langle\Gamma(s, \cdot +X_s) , \varphi\rangle + \langle \nabla F(s,\cdot +X_s) , \varphi\rangle] dW_s \nonumber
\\&+ \int_0^t [\langle \nabla\Gamma(s,\cdot +X_s) , \varphi\rangle  +  \langle A(s,\cdot+X_s) , \varphi\rangle ] ds\nonumber
\\ &+\int_0^t +  \langle\L^X_{s} F(s,\cdot +X_s) , \varphi\rangle ds, \; \P-a.s..\label{eq:IW}
\end{align}
\end{prop}

\begin{remark}
As noted earlier, elements of $\mathbb{W}^{2,p}_{\mathcal{P},q}$ are predictable with respect to $(\mathcal{F}^W_t)_{t\in [0,T]}$ the natural filtration of $W$. However, by definition of a weak solution to the SDE $(\mathcal{F}^W_t)_{t\in [0,T]}$-predictable processes are also $(\mathcal{F}_t)_{t\in [0,T]}$-predictable.
\end{remark}

\begin{remark}
Note that for any $\varphi$ in $L^{\bar{p}}(\mathbb{R}^d)$, the stochastic process $s \mapsto \langle\Gamma(s, \cdot +X_s), \varphi \rangle$ is square integrable so that the stochastic integral of this process against the Brownian motion is well-defined. The same comment implies that all the integrals involved in Relations \eqref{eq:decF}-\eqref{eq:IW} are well-defined. We also would like to point out that contrary to the original formula in \cite{ItoKrylov} where the test functions $\varphi$ are assumed to be infinitely differentiable, the regularity assumption on our processes allows us to consider only $L^{\bar{p}}$ test functions.
\end{remark}

With the previous results at hand we can now state and prove our main result, namely a \textit{It\^o-Wentzell-Tanaka trick}.

\begin{theorem}\label{thm:Main}
Assume that $f\in \D^{1,0,p}_q$ and that Assumption \ref{asm:fadpat} is in force. Let $(F,Z)$ be the unique strong solution to \eqref{eq:BSPDE}. Then we have, 
\begin{align}
\int_0^T f(s,X_s) ds =& -F(0,X_0) - \int_0^T \left(\nabla F(s,X_s) + Z(s,X_s)\right)dW_s \nonumber
\\ &- \int_0^T \nabla Z(s,X_s) ds, \; \P-a.s..\label{eq:SIWT}
\end{align}
\end{theorem}

\begin{proof}
It follows from the Itô-Wentzell formula from Proposition \ref{prop:ItoWentzell} that, $\forall \varphi\in L^{\bar p}(\mathbb{R}^d)$,
\begin{align}
\int_0^T \langle f(s,\cdot +X_s) &, \varphi \rangle ds \nonumber \\
=& -\langle F(0,\cdot+ X_0) , \varphi \rangle - \int_0^T \left( \langle \nabla F(s,\cdot + X_s) , \varphi \rangle + \langle Z(s,\cdot+X_s) , \varphi \rangle\right)dW_s \nonumber
\\ &- \int_0^T \langle \nabla Z(s,\cdot + X_s)  , \varphi \rangle ds, \; \P-a.s..\label{eq:WIWT}
\end{align}
Let us remark that by Theorem \ref{prop:ExistAdapt} and a Sobolev embedding, $F,Z\in L^q_{\mathcal{P}}([0,T];L^p(\Omega;\mathcal{C}^{1,\alpha}(\mathbb{R}^d)))$ for a certain $\alpha>0$. 
We choose $\varphi=\theta^\varepsilon$, $\varepsilon>0$ a mollifier in  Equation \eqref{eq:WIWT}.
For any positive $\varepsilon$ we have
\begin{align}
\int_0^T f^{\varepsilon}(s,X_s) ds &= -F^{\varepsilon}(0,X_0) - \int_0^T\left( \nabla F^{\varepsilon}(s,X_s) + Z^{\varepsilon}(s,X_s)  \right)dW_s \nonumber
\\ &\hspace{1.15em} - \int_0^T \nabla Z^{\varepsilon}(s,X_s) ds, \; \P-a.s.,\label{eq:MollifItoTanaka}
\end{align}
where we denote $G^{\varepsilon}(t,x) = \langle G(t,\cdot),\theta^{\varepsilon}(x-\cdot)\rangle$ for $G = f,F,\nabla F, \nabla Z$. We remark that, given a function $G\in L^q_{\mathcal{P}}([0,T];L^p(\Omega;\mathcal{C}^{0,\alpha}_b(\mathbb{R}^d)))$ it holds that
\begin{align*}
\mathbb{E}\left[\int_0^T \left| G^\varepsilon(s,X_s)-G(s,X_s)\right|ds\right]\leq& \left(\int_0^T \left(\mathbb{E}\left[\int_{\mathbb{R}^d} |G(s,x+X_s)-G(s,X_s)|\theta^{\varepsilon}(x) dx\right] \right)^{q} ds\right)^{1/q}
\\ \leq& \left(\int_0^T \mathbb{E}\left[\|G(s,\cdot)\|_{\mathcal{C}^{0,\alpha}_b(\mathbb{R}^d)}^p\right]^{q/p} ds\right)^{1/q}\left(\int_{\mathbb{R}^d} |x|^{\alpha}\theta^{\varepsilon}(x) dx \right)
\\ \leq& C\|G\|_{L^q([0,T];L^p(\Omega;\mathcal{C}^{0,\alpha}_b(\mathbb{R}^d)))} \varepsilon^{\alpha} \underset{\varepsilon\rightarrow 0}{\longrightarrow} 0.
\end{align*}
Thus, each term from the right-hand side of \eqref{eq:MollifItoTanaka} converges to the corresponding value. In order to handle with the term in the left-hand side, we have to prove that the integral $I$ defined by
\begin{equation*}
I(x) := \int_0^T f(s,x+X_s)ds,
\end{equation*}
is continuous, $\P-a.s.$. This comes from the fact that $I$ belongs to $\W^{1,p}$. Indeed, thanks to \eqref{eq:WIWT}, Itô's isometry, a change of variable and Jensen's inequality, we have that
\begin{align*}
\|I\|_{\W^{1,p}} \leq &  \|F(0,\cdot+X_0)\|_{\W^{1,p}} 
\\ &+2\left( \int_0^T \|\nabla F(s,\cdot+X_s)\|^2_{\W^{1,p}} + \|Z(s,\cdot+X_s)\|^2_{\W^{1,p}}ds \right)^{1/2}
\\ & + \int_0^T \|\nabla Z(s,\cdot+X_s)\|_{\W^{1,p}} ds
\\ \leq &  \|F(0,\cdot)\|_{\W^{1,p}}+C\left( \int_0^T \| F(s,\cdot)\|^q_{\W^{2,p}} + \|Z(s,\cdot)\|^q_{\W^{1,p}}ds \right)^{1/q}
\\ & + \left(\int_0^T \|Z(s,\cdot)\|_{\W^{2,p}}^q\right)^{1/q} ds.
\end{align*}
Since $F,Z \in \W^{1,2,p}_q$, we deduce that $I\in\W^{1,p}$. By the Sobolev embedding $\mathcal{C}^{0,\alpha}(\mathbb{R}^d) \subset W^{1,p}(\mathbb{R}^d)$, we deduce that $I$ is $\P$-a.s. continuous. Thus, we have, by using Fubini's theorem,
\begin{equation*}
\left|\int_0^T [f^{\varepsilon}(s,X_s)-f(s,X_s)]ds\right| = \left| \langle [I(\cdot) -I(0)],\theta^{\varepsilon}\rangle \right| \underset{\varepsilon\rightarrow 0}{\longrightarrow} 0, \; \P-a.s.,
\end{equation*}
which concludes the proof.
\end{proof}

\begin{remark}
If $f$ and $b$ are deterministics, then, the BSPDE \eqref{eq:BSPDE} reduces to a PDE that is $Z\equiv 0$. Hence, $\nabla Z \equiv 0$ and we recover the formula of \cite{krylov2005strong}. In particular, if $b$ does not depend on $\omega$ and if $f$ is random then the gain/loss of regularity than one could obtain by using the It\^o-Tanaka-Wentzell trick compared to the It\^o-Tanaka trick is completely contained in the regularity of $Z$ and its gradient. 
\end{remark}

\subsection{Example: Smooth perturbations of a Brownian motion}

Let us consider the following functional
\begin{equation}\label{eq:IntegSmoothBrow}
I(x) : = \int_0^T f(t,x+W_t+Y_t)dt, \quad x \in \real^d,
\end{equation}
where $f\in L^q([0,T];L^{p}(\mathbb{R}^d))$, $W$ is a standard brownian motion and $Y$ is a stochastic process adapted to the filtration $\left(\mathcal{F}_t^W\right)_{t\in[0,T]}$ such that
\begin{equation*}
\left\|\frac {\|D_tY_{\cdot}\|_{L^p(\Omega)}}{|\cdot-t|^{3/2}}\right\|_{L^{\bar{q}}([t,T])}\in L^q([0,T]).
\end{equation*}
We remark that when $Y$ is of the form
\begin{equation*}
Y_t = \int_0^t h_s ds
\end{equation*}
where $h$ is a stochastic process adapted to $\left(\mathcal{F}_t^W\right)_{t\in[0,T]}$, one can apply Girsanov's theorem and remove $Y$ from \eqref{eq:IntegSmoothBrow}. This means that $Y$ can not pertubate the regularization effect of $W$. That is, one could apply the Itô-Tanaka trick (\textit{i.e.} the deterministic version as provided in \cite{flandoli2010well}) to $I$ under a probability measure $\mathbb{Q}$ and obtain
\begin{equation*}
I = -F(0,x) - \int_0^T \nabla F(t,x+W_t)dW_t, \; \mathbb{Q}-a.s.
\end{equation*}
without having to consider the extra terms involving Malliavin derivatives. Here, our objective is to show that Theorem \ref{thm:Main} is consistent with those arguments and that we recover the same type of result.

By considering the random function $\tilde{f} : = f(t,\cdot+Y_t)$, the It\^o-Tanaka-Wentzell trick gives is the following expression of $I$
\begin{equation*}
I = -\tilde{F}(0,x+Y_0) - \int_0^T\left(\nabla \tilde{F}(t,x+W_t)+\tilde{Z}(t,x+W_t)\right)dW_t -\int_0^T \nabla \tilde{Z}(t,x+W_t)dt,
\end{equation*}
with
\begin{equation*}
\tilde{F}(t,x) = \E\left[-\int_t^T P_{t,s}\tilde{f}(s,x) ds\Big\vert\mathcal{F}_t\right],
\end{equation*}
and
\begin{equation*}
\tilde{Z}(t,x) = \E\left[-\int_t^T P_{t,s}D_t\tilde{f}(s,x) ds\Big\vert\mathcal{F}_t\right].
\end{equation*}
We notice that the Malliavin derivative of $f$ implies, \textit{a priori}, a loss of regularity compared to the case where $f$ is deterministic since
\begin{equation}\label{eq:LossDeriv}
D_t\tilde{f}(s,x) = \nabla f(s,x+ Y_s) \cdot D_tY_s.
\end{equation}
However, this is not the case as proved in the following Lemma.
\begin{lemma}\label{lmm:DbleHeat}
Let $f\in L^{q}([0,T];L^p(\mathbb{R}^d))$. There exists a constant $C>0$ such that the following estimate holds
\begin{equation*}
\left\|\int_t^T P_{t,s}f(s,x + Y_s) D_tY_s ds\right\|_{\mathbb{W}^{3,p}_q} \leq C \|f\|_{L^{q}([0,T];L^p(\mathbb{R}^d))}.
\end{equation*}
\end{lemma}
\begin{proof}
We have, by a classical estimate on the heat semigroup, a change of variable and Hölder's inequality,
\begin{align*}
\left\|\int_t^TP_{t,s}f(s,\cdot+ Y_s) D_tY_s ds \right\|_{\mathbb{W}^{3,p}} &\leq C\int_t^T \frac {\|D_tY_s\|_{L^{p}(\Omega)}} {(s-t)^{3/2}}\| f(s,\cdot)\|_{L^p(\mathbb{R}^d)}ds
\\&\leq C\left(\int_t^T\frac {\|D_tY_s\|^{\bar{q}}_{L^{p}(\Omega)}} {(s-t)^{3\bar{q}/2}}ds \right)^{1/\bar{q}}\|f\|_{L^{q}([0,T]; L^p(\mathbb{R}^d))},
\end{align*}
which yields the estimate.
\end{proof}

It follows from the previous Lemma that, in fact, the additional terms coming from the It\^o-Tanaka-Wentzell trick are at least as smooth as the ones from the Itô-Tanaka trick. Thus, in this example which can be considered at the interface between the case where $f$ is deterministic and the case where it is random, our formula recovers the regularization effect.

\begin{remark}
In the case where $Y_t = -W_t$, we should not expect any regularization from the brownian motion (as mentioned in the introduction). We observe that, in this context, $D_t Y_s = -\bold{1}_{[0,s]}(t)$. Thus, we are not able to apply Lemma \ref{lmm:DbleHeat} and equation \eqref{eq:LossDeriv} shows that we lose one degree of regularity. Then the It\^o-Tanaka-Wentzell trick does not bring any regularization effect.
 \end{remark}
 
The classical It\^o-Tanaka trick find several applications as stated in the introduction including the strong uniqueness of solution to SDEs. These applications will be studied by the authors in a separated work.

\end{document}